\documentclass[twoside,reqno]{amsart}

\usepackage{hyperref}
\usepackage{amsmath}
\usepackage{amssymb,latexsym,amstext,amsthm}
\usepackage{verbatim} 
\usepackage{eucal} 
\usepackage{color}
\usepackage{pdfsync}

\theoremstyle{plain}
\newtheorem{thm}[equation]{Theorem}
\newtheorem{pro}[equation]{Proposition}
\newtheorem{cor}[equation]{Corollary}
\newtheorem{lem}[equation]{Lemma}

\theoremstyle{definition}
\newtheorem{exa}[equation]{Example}
\newtheorem{con}[equation]{Convention}
\newtheorem{DEF}[equation]{Definition}
\newtheorem{rem}[equation]{Remark}

\numberwithin{equation}{section}

\def\rd{\dot R}
\def\ad{\dot A}
\def\w{\mathcal{W}}
\def\v{\mathcal{V}}
\def\vd{\dot \v}
\def\vt{\tilde{\v}}
\def\wt{\widetilde{\w}}
\def\bes{\mathfrak{B}}
\def\bbbz{{\mathbb Z}}
\def\G{\widetilde{G}}

\def\fm{(\cdot,\cdot)}
\def\Aut{\mathrm{Aut}}
\def\sgn{\mathrm{sgn}}
\def\andd{\quad\mathrm{and}\quad}
\def\spani{\mathrm{span}}
\def\la{\langle}
\def\ra{\rangle}
\def\sub{\subseteq}
\def\alt{\mathrm{Alt}}
\def\talt{\widetilde\alt}

\def\a{\alpha}
\def\b{\beta}
\def\ep{\epsilon}
\def\sg{\sigma}
\def\lam{\lambda}
\def\Lam{\Lambda}
\def\d{\delta}
\def\xt{\widetilde{x}}
\def\supp{\hbox{supp}}
\def\vep{\varepsilon}
\def\bc{\mathcal{B}}
\def\pc{\mathcal{P}}

\begin{document}



\title[Affine Reflection Systems of Type $A_1$]{Weyl Groups Associated with Affine Reflection Systems of Type $A_1$\\ (Coxeter Type Defining Relations)}


\author{Saeid Azam and Mohammad Nikouei}
\thanks{S. Azam: School of Mathematics, Institute for Research in Fundamental Sciences (IPM), P.O.
Box 19395-5746, Tehran, Iran, and Department of Mathematics, University of Isfahan, P.O. Box
81745-163, Isfahan, Iran (Corresponding Author).\\
Email: azam@sci.ui.ac.ir}
\thanks{M. Nikouei: Department of mathematics, University of Isfahan, P.O.Box. 81745-163, Isfahan, Iran.\\
Email:  nikouei@sci.ui.ac.ir, nikouei2002@yahoo.com}

\subjclass[2010]{Primary 20F55; Secondary 17B67.}
\keywords{Affine reflection system, Weyl group, hyperbolic Weyl group, extended affine Weyl group.}

\maketitle

\begin{abstract}
In this paper, we offer a presentation for the Weyl group of an affine reflection system $R$ of type $A_1$
as well as a presentation for the so called {\it hyperbolic} Weyl group associated with an affine reflection system of type $A_1$.
Applying these presentations to extended affine Weyl groups, and using a description of the center of the hyperbolic Weyl group,
we also give a new finite presentation for an extended affine Weyl group of
type $A_1$. Our presentation for the (hyperbolic) Weyl group of an affine reflection system of type $A_1$ is the
first non-trivial presentation given in such a generality, and can be considered as a model for other types.
\end{abstract}


\setcounter{section}{-1}
\section{\bf Introduction}\label{introduction}
Weyl groups, as reflection groups, always give a geometric meaning to underlying structures such as root systems, Lie algebras and Lie groups. Thus to get a ``good" perspective of these structures, one needs to have a better understanding of their Weyl groups. The present work is dedicated to the study of some new presentations for (hyperbolic) Weyl groups associated with affine reflection systems of type $A_1$.

Affine reflection systems are the most general known extensions of finite and affine root systems introduced by E. Neher and O. Loos in \cite{LN2}.
They include locally finite root systems, toroidal root systems, extended affine root systems, locally extended affine root systems, and root systems extended by an abelian group.
In 1985, K. Saito \cite{S} introduced the notion of an extended affine root system as a generalization of finite and affine root systems.
For a systematic study of such root systems the reader is referred to \cite{AABGP}. Another generalization of finite root systems are
locally finite root systems, which one can find a complete account  of them in \cite{LN1}.
Root systems extended by an abelian group and locally extended affine root systems, introduced in \cite{Y1} and \cite{Y2},
are two other generalizations which include extended affine root systems and locally finite root systems, respectively.

In \cite{AYY}, the authors introduce an equivalent definition for
an affine reflection system (see Definition \ref{ARS def}) which
we will use it here. In finite and affine cases, the corresponding
Weyl groups are fairly known. In particular, they are known to be Coxeter groups and that through their actions implement a specific geometric and
combinatorial structure on their underlying root systems. In
extended affine case, however, it is known that if the nullity is
bigger than one, then the corresponding hyperbolic Weyl groups,
called {\it extended affine Weyl groups}, are not  any more
Coxeter groups, see \cite[Theorem 3.6]{HCox}. Here we record some advances made on
presentations of (hyperbolic) Weyl groups of certain subclasses of
affine reflection systems. As the starting point in this direction, we can name the works of \cite{K}, \cite{AEAWG} and \cite{A1} which consider certain presentationes for toroidal Weyl groups and extended affine Weyl groups. In \cite{ST}, the authors give a generalized
Coxeter presentation for extended affine Weyl groups of nullity
$2$. In \cite{ASFP}, \cite{ASPEAWG}, \cite{ASPCA1}, \cite{HA1} and
\cite{HCox}, the authors offer some new presentations for extended
affine Weyl groups, where they achieve this by a group theoretical point of view analysis of the structure and in particular center of extended affine Weyl groups. To find a
comprehensive account on the structure of extended affine Weyl
groups, the reader is referred to  \cite{MS}, \cite{AEAWG},
\cite{ASEAWGA1}, \cite{Hab} and \cite{HA1}.

In the study of groups associated with affine reflection systems
and other extensions of finite and affine root systems, type
$A_1$ has always played a special role and usually is
considered as a model for other types. A philosophical justification for this is that any (tame) affine reflection system can be considered as a  union of a family of affine reflection systems of type $A_1$.

In this work, we study two groups associated with an affine
reflection system $R$. The first one, defined in \cite{LN2}, is
called the {\it Weyl group} of $R$ which we show by $\w$, and
the other, denoted by $\wt$, is defined when the ground abelian
group is a torsion-free abelian group.  We call $\wt$, the {\it
hyperbolic} Weyl group of $R$, see Definition \ref{wt} (compare
with \cite[Definition 3.1]{HA1}). One should beware that the
notion of a hyperbolic Weyl group is a generalization of the
definition of an extended affine Weyl group.

Let $R$ be some affine reflection system of type $A_1$ in an abelian group $A$.  In Section \ref{ARS Def},
we give preliminary definitions as well as recording
some results and facts on affine reflection systems. In Sections \ref{Presentation} and \ref{hweyl},
we offer two presentations for $\w$ and  $\wt$, respectively, see Theorems \ref{W(V) AP} and \ref{W AP}.
A quick look at these presentations shows that both $\w$ and $\wt$
have soluble word problems. Also, using the description of the center of $\wt$ given in Proposition \ref{SM lem},
one perceives that $\wt$ is a central extension of $\w$.
In Theorem \ref{W W(V) CentP}, we prove under a very particular set of conditions that
the existence of a presentation for $\w$ is equivalent to
the existence of a presentation for $\wt$. This is the main tool in Section \ref{RedPre}
for obtaining a presentation for $\wt$ for the case in which the given  reflectable base is \textit{elliptic-like}.
In Section \ref{RedPre}, we assume that $R$ is an extended affine root system of type $A_1$, i.e.,
$R$ is an affine reflection system in a
free abelian group $A$ of rank $\nu+1$. Then we offer two finite presentations, one for $\w$ in
Theorem \ref{Baby W pre} and the other for the so called {\it baby} extended affine Weyl group $\wt$ of
type $A_1$ in Proposition \ref{Baby Tor Pre}. The latter presentation has $\nu+1$ generators and
$\big(\nu(\nu+1)/2\big)+\nu+1$ relations, where relations consist of $\nu(\nu+1)/2$ central elements and $\nu+1$ involutions.
The paper is concluded with an appendix, Section \ref{geometric}, in which we provide a geometric approach to the proof of Theorem \ref{Baby W pre}.

Some of our results in this work are suggested by running
a computer programming designed specifically for calculating certain relations among elements of an extended affine Weyl group of type $A_1$.
This program consists of several algorithms, written in Visual Basic .Net. The interested reader can find the program and its source code at http://sourceforge.net/projects/central-exp/files/.

\section{\bf Affine reflection system and their Weyl groups}\setcounter{equation}{0}\label{ARS Def}
In this section, we recall the definition and some properties of affine reflection systems. Affine reflection systems
are introduced by E. Neher and  O. Loos  in \cite{LN2}. Here we follow an equivalent definition given in \cite{AYY}.
Let $A$ be an abelian group. By a symmetric form on $A$, we mean a symmetric
bi-homomorphism $\fm :A\times A\longrightarrow\mathbb{Q}$.
The radical of the form is the subgroup $A^0=\{\alpha\in A\;|\;(\alpha,A)=0\}$ of $A$.
Also, let $A^\times=A\setminus A^0$, $\bar{A}=A/A^0$ and $\;\bar{}:A\longrightarrow\bar{A}$ be the canonical map.
The form $\fm$ is called {\it positive definite (positive semidefinite)} if $(\a,\a)>0$ $((\a,\a)\geq 0)$ for all $\a\in A\setminus\{0\}$. If $\fm$ is positive semidefinite, then it is easy to see that
$$A^0=\{\alpha\in A\;|\;(\alpha,\alpha)=0\}.$$
Assume from now on that $\fm$ is positive semidefinite on $A$.
For a subset $B$ of $A$, let $B^\times=B\setminus A^0$ and $B^0=B\cap A^0$.
For $\alpha,\beta\in A$, if $(\alpha,\alpha)\not= 0$, set $(\beta,\alpha^\vee):=
2(\beta,\alpha)/(\alpha,\alpha)$ and if $(\alpha,\alpha)=0$, set $(\beta,\alpha^\vee):=0$.
A subset $X$ of $A$ is called {\it connected} if it cannot be written as a disjoint union of two nonempty
orthogonal subsets. The form $\fm$ induces a unique form on $\bar{A}$ by
$$(\bar{\alpha},\bar{\beta})=(\alpha,\beta)\quad\mathrm{for}\;\alpha,\beta\in A.$$
This form is positive definite on $\bar{A}$. Thus, $\bar{A}$ is a torsion-free group.
For a subset $S$ of $A$, we denote by $\la S\ra$, the subgroup generated by $S$.
Here is the definition of an affine reflection system given in \cite{AYY}.

\begin{DEF}\label{ARS def}
Let $A$ be an abelian group equipped with a nontrivial symmetric
positive semidefinite form $\fm$. Let $R$ be a subset of $A$. The triple $(A,\fm,R)$, or $R$ if
there is no confusion, is called an {\it affine reflection system} if it satisfies the following 3
axioms:
\begin{itemize}
\item[(R1)] $R=-R$,
\item[(R2)] $\langle R\rangle=A$,
\item[(R3)] for $\alpha\in R^\times$ and $\beta\in R$, there exist $d, u\in\mathbb{Z}_{\geq0}$ such that
$$(\beta+\mathbb{Z}\alpha)\cap R=\{\beta-d\alpha,\dots,\beta+u\alpha\}\quad\mathrm{and}\quad d-u=(\beta,\alpha^\vee).$$
\end{itemize}
The affine reflection system $R$ is called {\it irreducible} if it
satisfies
\begin{itemize}
\item[(R4)] $R^\times$ is connected.
\end{itemize}
Moreover, $R$ is called {\it tame} if
\begin{itemize}
\item[(R5)] $R^0\subseteq R^\times-R^\times$ (elements of $R^0$ are non-isolated).
\end{itemize}
Finally $R$ is called {\it reduced} if it satisfies
\begin{itemize}
\item[(R6)] $\alpha\in R^\times\Rightarrow 2\alpha\not\in R^\times$.
\end{itemize}
Elements of $R^\times$ (resp. $R^0$) are called {\it non-isotropic
roots} (resp. {\it isotropic roots}).
An affine reflection system $(A,\fm,R)$ is called a {\it locally finite root system} if $A^0=\{0\}$.
\end{DEF}
Here, we recall some results about the structure of affine reflection systems from \cite{AYY}.
The image of $R$ under $\bar{\;}$ is shown by $\bar{R}$. Since the form is nontrivial, $A\not=A^0$.
It then follows from axioms that
$$0\in R.$$

\begin{pro}\cite[Corollary 1.9]{AYY}\label{LFRA}
If $(A,\fm,R)$ is an affine reflection system, then $(\bar{R},\fm,\bar{A})$ is a locally finite root system.
In particular, if $R$ is irreducible, the induced form on $\bar{\v}:=\mathbb{Q}\otimes_\mathbb{Z}\bar{A}$ is positive definite.
\end{pro}

{\it Type} of $R$ is defined to be the type of $\bar{R}$. Since this work is devoted to the study of Weyl groups  associated with affine reflection systems
of type $A_1$, {\it for the rest of this work we assume that $(A,\fm, R)$ is a tame irreducible affine reflection system of type $A_1$}.
By \cite[Theorem 1.13]{AYY}, $R$ contains a finite root system $\rd=\{0,\pm\ep\}$ and a subset $S\sub R^0$, such that
\begin{equation}\label{type-A_1}
R=(S+S)\cup(\rd+S),
\end{equation}
where $S$ is a {\it pointed reflection subspace of} $A^0$, in the sense that it satisfies,
\begin{equation}\label{eq1}
0\in S,\quad S\pm 2S\sub S,\andd \la S\ra=A^0.
\end{equation}
In fact by \cite[Theorem 1.13]{AYY}, any tame irreducible affine reflection system of type $A_1$ arises this way.

\begin{rem}
By \cite[Remark 1.16]{AYY}, the definition of a tame irreducible affine reflection system of type $A_1$
is equivalent to the definition of an irreducible root system extended by an abelian group, of type $A_1$,
defined by Y. Yoshii \cite{Y1}.
If $A$ is a free abelian group of finite rank and $R$ is a
tame irreducible affine reflection system,
we may identify $R$ with
the subset $1\otimes R$ of $ R\otimes _\mathbb{Q}A$.
Then $R$ is isomorphic to an extended affine root system in the sense of \cite{AABGP}.
\end{rem}

Let $\Aut(A)$ be the group of automorphisms of $A$. For $\a\in A$, one defines $w_\a\in\Aut(A)$ by
$$w_\a(\b)=\b-(\b,\a^\vee)\a.$$ We call $w_\a$ the {\it reflection based on} $\a$, since it sends $\a$ to $-\a$ and
fixes pointwise the subgroup $\{\b\in A\mid (\b,\a)=0\}$. Note that if $\a\in A^0$, then according to our convention,
$(\b,\a^\vee)=0$ for all $\b$ and so $w_\a=\hbox{id}_A$.
For a subset $S$ of $R$, the subgroup of $\Aut(A)$ generated by $w_\a$, $\a\in S$, is denoted by $\w_S$.
It now allows us to define the Weyl group of the affine reflection system $R$.

\begin{DEF}\label{def1}
$\w:=\w_R$ is called the {\it Weyl group} of $R$.
\end{DEF}

In a similar way, one defines $w_{\bar{\alpha}}$, the reflection based on $\bar{\alpha}$,
for $\bar{\alpha}\in\bar{A}$, and $\bar{\w}$, the Weyl group of $\bar{R}$, which is a subgroup of
$\mathrm{Aut}(\bar{A})$ generated by $w_{\bar{\alpha}}$, $\bar{\alpha}\in\bar{R}$.

Next we associate another reflection group with $R$. To do so, from here until the end of this section we assume that $A$ is a torsion-free abelian group. To justify this assumption on $A$, we recall that the root system of any invariant affine reflection algebra, over a field of characteristic zero,
is an affine reflection system whose $\bbbz$-span is a torsion-free abelian group.
Indeed, affine reflection systems corresponding to invariant affine reflection algebras, are contained in the dual space of the so called \textit{toral subalgebras}. Thus the root system is a subset of a torsion-free group.

According to \cite[Remark 1.6 (ii)]{AYY}, one can transfer the form on $A$ to the vector space
$\v:=\mathbb{Q}\otimes_\mathbb{Z}A$. Then the radical of the form can be
identified with $\v^0:=\mathbb{Q}\otimes A^0$ and one may conclude that the form on $\v$ is a
positive semidefinite symmetric bilinear form.
Then $A$ and $R$ can be identified with subsets $1\otimes A$ and $1\otimes R$ of $\v$, respectively.
Set $\dot{A}=\mathbb{Z}\ep$ and $\vd:=\mathbb{Q}\otimes\dot{A}$, then we have $\v=\vd\oplus\v^0$.

Since $\la S\ra=A^0$, it follows that $S\equiv 1\otimes S$ spans $\v^0$.
So $S$ contains a basis $\mathfrak{B}^0=\{\sg_j\mid j\in J\}$ of $\v^0$. Therefore, if $\dot{\mathfrak{B}}=\{\ep\}$, then
$\mathfrak{B}=\dot{\mathfrak{B}}\cup\mathfrak{B}^0$ forms a basis for $\v$.
For $j\in J$, define
$\lam_j$ as an element of the dual space $(\v^0)^\star$ of $\v^0$ by $\lam_j(\sg_i)=\d_{ij}$.
We call $(\mathfrak{B}^0)^\star:=\{\lam_j\mid j\in J\}$ the {\it dual basis} of $\mathfrak{B}^0$.
Let $(\v^0)^\dag$ be the subspace of $(\v^0)^\star$ spanned by the dual basis $(\mathfrak{B}^0)^\star$. Namely
$$(\v^0)^\dag:=\sum_{j\in J}{\mathbb Q}\lam_j.$$
Note that $(\v^0)^\dag$, as a subspace of $(\v^0)^\star$,
can be identified with the restricted dual space $\sum_{j\in J}({\mathbb Q}\sg_j)^\star$
of $\v^0$, with respect to the basis $\{\sg_j\mid j\in J\}$.
We now set
$$\vt:=\v\oplus(\v^0)^\dag=\vd\oplus\v^0\oplus (\v^0)^\dag.$$
We then extend the form $\fm$ on $\v$ to a non-degenerate form on $\vt$, denoted again by $\fm$ as  follows,
\begin{itemize}
\item $(\vd,(\v^0)^\dag):=((\v^0)^\dag,(\v^0)^\dag):=\{0\}$,
\item $(\sigma,\lambda):=\lambda(\sg),\hbox{ for }\sg\in\v^0\hbox{ and }\lam\in (\v^0)^\dag$.
\end{itemize}
We call $\vt$ the {\it hyperbolic extension} of $\v$, with respect to the basis $\{\sg_j\mid i\in J\}$.
For each $\alpha\in\v$, we define $w_\a\in\Aut(\vt)$ by
$$w_\alpha(\beta)=\beta-(\beta,\alpha^\vee)\alpha,$$
where $(\beta,\alpha^\vee):=2(\beta,\alpha)/(\alpha,\alpha)$, if $(\a,\a)\not=0$ and $(\b,\a^\vee):=0$, if $(\a,\a)=0$.
Clearly, $w_\a$ is a reflection with respect to the vector space $\vt$.
For a subset $S$ of $\v$, we define $\wt_S$ to be the subgroup of $\Aut(\vt)$ generated by $w_\a$, $\a\in S$.
\begin{DEF}\label{wt}
We call $\wt:=\wt_R$, the {\it hyperbolic Weyl group} of $R$.
\end{DEF}

We note that for $w\in\wt$, $\a\in\v$ and $\b,\gamma\in\vt$, we have
\begin{equation}\label{Conj Pro}
w_\a^2=1,\qquad(w\b,w\gamma)=(\b,\gamma),\andd ww_\alpha w^{-1}=w_{w(\alpha)}.
\end{equation}

Since $A$ as a torsion-free abelian group is embedded in $\v$,
the restriction of elements of $\wt$ to $\v$ induces an epimorphism $\varphi:\wt\longrightarrow\w$.

\section{\bf Alternating presentation}\setcounter{equation}{0}\label{Presentation}
Let $A$ be an arbitrary abelian group and $(A,\fm,R)$ be a fixed tame irreducible affine reflection system of type $A_1$.
As we have seen in (\ref{type-A_1})
$$R=(S+S)\cup(\dot{R}+S),$$
where $S$ is a subset of $R^0$ satisfying (\ref{eq1}), and $\rd=\{0,\pm\ep\}$. We may assume that
\begin{equation}\label{form}
(\epsilon,\epsilon)=2.
\end{equation}
Note that $A=\dot{A}\oplus A^0$, $\la\rd\ra=\dot A$ and $\la R^0\ra=\la S\ra=A^0$. Let
$p:A\rightarrow A^0$ be the projections onto $A^0$.

For each $\alpha\in A$, we have
$$\alpha=\sgn(\alpha)\epsilon+p(\alpha),$$
where $\sgn:A\longrightarrow{\mathbb Z}$ is a group epimorphism. Clearly, each $\a\in A$ is uniquely
determined by $\sgn(\a)$ and $p(\a)$.
Then for $\a\in R^\times$ and $\b\in R$, we have
\begin{equation}\label{form-sgn}
(\beta,\alpha^\vee)=(\beta,\alpha)=2\sgn(\beta)\sgn(\alpha).
\end{equation}

\begin{lem}\label{action lem}
Let $w:=w_{\alpha_1}\cdots w_{\alpha_t}\in\w$.
Then for $\beta\in R^\times$, we have
$$\sgn(w\b)=(-1)^t\sgn(\beta)\andd p(w\b)=p\big(\beta-2(-1)^t\sgn(\b)\sum_{i=1}^t(-1)^i\sgn(\alpha_i){\alpha_i}\big).$$
\end{lem}

\begin{proof}
First, let $t=1$, then $w\b=w_\alpha(\b)=\b-(\b,\a)\a$, so by (\ref{form-sgn}) we have
\begin{equation}
\sgn(w\b)=sgn(\b)-2\sgn(\a)\sgn(\b)\sgn(\a)=-\sgn(\b)=(-1)^1\sgn(\b),
\end{equation}
and
\begin{equation}
p(w\b)=p(\b-(\b,\a)\a)=p(\b-2\sgn(\a)\sgn(\b)\a).
\end{equation}
So the result holds for $t=1$.
Now, assume that the statement holds for $t\leq k$. Now if $w'=w_{\a_2}\cdots w_{\a_{k+1}}$, then for $t=k+1$, we have
$w\b=w'\b-(w'\b,\a_1)\a_1$, so
\begin{eqnarray*}
\sgn(w\b)&=&\sgn(w'\b)-2\sgn(w'\b)\sgn(\a_1)\sgn(\a_1)\\
&=&-\sgn(w'\b)\\
&=&-(-1)^k\sgn(\b)\\
&=&(-1)^{k+1}\sgn(\b).
\end{eqnarray*}
Also
\begin{eqnarray*}
p(w\b)&=&p(w'\b)-2\sgn(w'\b)\sgn(\a_1)p(\a_1)\\
&=&p\big(\b-2(-1)^k\sgn(\b)\sum_{i=2}^{k+1}(-1)^{i-1}\sgn(\a_i)\a_i\big)
-2(-1)^k\sgn(\b)\sgn(\a_1)p(\a_1)\\
&=& p\big(\b-2(-1)^{k+1}\sgn(\b)\sum_{i=1}^{k+1}(-1)^{i}\sgn(\a_i)\a_i\big).
\end{eqnarray*}
\end{proof}

\begin{pro}\label{cor1}
For $\a_1,\ldots,\a_n\in R^\times$, we have $w:=w_{\a_1}\cdots w_{\a_n}=1$ in $\w$, if and
only if $n$ is even and
$$\sum_{i=1}^n(-1)^i\sgn(\a_i)p(\a_i)=0.$$
In particular, if $n$ is odd, then $w^2=1$.
\end{pro}

\begin{proof} We first prove the first assertion in the statement.
Since $R^\times$ generates $A$, it is enough to show that, for each $\b\in R^\times$,
$w\b=\b$ if and only if $n$ is even and $\sum_{i=1}^{n}(-1)^i\sgn(\a_i)p(\a_i)=0$.
However, since the maps $p$ and $\sgn$ determine $w\b$ uniquely, the result immediately follows from
Lemma \ref{action lem}.

Now assume that $n$ is odd, then for $w^2=w_{\a_1}\cdots w_{\a_n}w_{\a_1}\cdots w_{\a_n}$, we have
\begin{eqnarray*}
&&\sum_{i=1}^{n} (-1)^i\sgn(\a_{i})p(\a_{i})
+\sum_{i=n+1}^{2n}(-1)^i\sgn(\a_{i-n})p(\a_{i-n})\\
&=&
\sum_{i=1}^{n}[(-1)^i+(-1)^{n+i}]\sgn(\a_{i})p(\a_{i})
\\
&=&0,
\end{eqnarray*}
where the last equality holds since $n$ is odd. So $w^2=1$ by the first assertion.
\end{proof}

Suggested by Proposition \ref{cor1}, we make the following definition.
\begin{DEF}\label{semi-relation}
Let $P$ be a subset of $R^\times$. We call a  $k$-tuple $(\a_1,\ldots,\a_k)$ of roots in $P$, an {\it alternating $k$-tuple in $P$} if $k$ is even
and $\sum_{j=1}^k(-1)^jsgn(\alpha_j)p(\alpha_j)=0$. We denote by $\alt(P)$, the set of all alternating $k$-tuples in $P$ for all $k$.
By Proposition \ref{cor1}, if $(\a_1,\ldots,\a_k)\in\alt(P)$, then $w_{\a_1}\cdots w_{\a_k}=1$ in $\w$.
\end{DEF}

\begin{cor}\label{action cor}
For $\alpha_1,\dots,\alpha_t\in R^\times$ we have, as elements of $\w$,
$$w_{\alpha_1}\cdots w_{\alpha_t}=w_{\alpha_1}\cdots w_{\alpha_{i-1}}w_{\alpha_{i+2}}w_{\alpha_{i+1}}w_{\alpha_i}w_{\alpha_{i+3}}\cdots w_{\alpha_t},$$
where $1\leq i\leq t-2$.
\end{cor}
\begin{proof} Clearly equality holds if and only if $w_{\a_i}w_{\a_{i+1}}w_{\a_{i+2}}=w_{\a_{i+2}}w_{\a_{i+1}}w_{\a_{i}}$ and in turn, this holds
if and only if $(w_{\a_i}w_{\a_{i+1}}w_{\a_{i+2}})^2=1$, which holds by Proposition \ref{cor1}.
\end{proof}

For our next result, we recall a definition from \cite[Defintion 1.19 (ii)]{AYY}.
First we recall that for a subset $S$ of $A$, $\w_S$ is by definition the subgroup of $\w$
generated by all reflections $w_\a$, $\a\in S$.

\begin{DEF}\label{reflectable def}
Let $P\subseteq R^\times$.
\begin{itemize}
\item[(i)] The set $P$ is called a {\it reflectable set}, if $\w_PP=R^\times$.
\item[(ii)] The set $P$ is called a {\it reflectable base}, if $P$ is a reflectable set
and no proper subset of $P$ is a reflectable set.
\end{itemize}
\end{DEF}
Obviously, if $P$ is a reflectable set, then $\w=\w_P$. In \cite{AYY}, reflectable sets and reflectable
bases are characterized for tame irreducible affine reflection systems of reduced types. As it is shown
in \cite{AEAWG}, one might find finite reflectable sets and bases for many interesting affine reflection systems, such as
extended affine root systems.

Here is our main result of this section which provides a presentation for
the Weyl group of an affine reflection system of type $A_1$.

\begin{thm}\label{W(V) AP}
Let $(A,\fm,R)$ be an affine reflection system of type $A_1$ and
$\Pi\subseteq R^\times$ be a reflectable set for $R$. Then, $\w$ is isomorphic to the group $G$,
defined by
\begin{itemize}
\item[] generators: $x_\alpha$, $\alpha\in\Pi$,
\item[] relations: $x_{\alpha_1}\cdots x_{\alpha_k}$, $(\a_1,\ldots,\a_k)\in\alt(\Pi)$.
\end{itemize}
\end{thm}
\begin{proof}
By definition of a reflectable set we have $\w=\la w_\a\mid \a\in\Pi\ra$.
Let $x_{\alpha_1}\cdots x_{\alpha_k}$ be a defining relation in $G$. Then by definition, $k$
is even, $\sum_{j=1}^k(-1)^j\sgn(\alpha_j)p(\alpha_j)=0$ and $\a_i\in\Pi$ for $1\leq i\leq k$. By Proposition \ref{cor1},
we have $w_{\alpha_1}\cdots w_{\alpha_k}=1$.
So, the assignment $x_\a\longmapsto w_\a$, $\a\in\Pi$ induces an epimorphism
$\phi:G\longrightarrow\w$. Let $x=x_{\alpha_1}\cdots x_{\alpha_k}\in\mathrm{Ker}\phi$, $\a_i\in\Pi$, then we have
$$1=\phi(x)=w_{\alpha_1}\cdots w_{\alpha_k}.$$ By Proposition \ref{cor1},
$k$ is even and $p(\sum_{j=1}^k(-1)^j\sgn(\alpha_j)\alpha_j)=0$. Therefore, we have
$(\a_1,\ldots,\a_k)\in \alt(\Pi)$ and so $\phi$ is an isomorphism.
\end{proof}

\begin{rem}
Let $\Pi'$ be obtained from $\Pi$ by changing the sign of elements of a subset of $\Pi$.
Then it is clear that $\Pi$ is a reflectable set
if and only if $\Pi'$ is a reflectable set. So
without loss of generality we may assume that
$\Pi\sub\epsilon+S$. Now if we set $\widehat{\Pi}=p(\Pi)$, then the above
presentation can be written in the form:
\begin{itemize}
\item[] generators: $x_\tau$, $\tau\in\widehat\Pi$,
\item[] relations: $x_{\tau_1}\cdots x_{\tau_k}$, $k$ is even and  $\sum_{j=1}^k(-1)^j\tau_j=0.$
\end{itemize}
\end{rem}

\section{\bf Hyper-alternating presentation}\setcounter{equation}{0}\label{hweyl}
Throughout this section, we assume that $A$ is a torsion-free abelian group. Let $(A,\fm,R)$ be an affine reflection system of type $A_1$. As in Section \ref{ARS Def}, let $\v=\mathbb Q\otimes_{\mathbb Z} A$ and consider the basis $\bes=\{\epsilon\}\cup\bes^0$ of $\v$ where $\bes^0=\{\sigma_j\mid j\in J\}$ is a basis of $\v^0$. Let $(\bes^0)^*=\{\lambda_j\mid j\in J\}$ be the dual basis of $\mathfrak{B}^0$ defined in Section \ref{ARS Def} and consider the corresponding hyperbolic extension $\vt=\v\oplus(\v^0)^\dag$ of $\v$.
Then as we have already seen, since $A$ is torsion-free, we are able to identify $R$ with $1\otimes R$ as a subset of
$\v\subseteq\vt$. This is done in order to study the hyperbolic Weyl group $\wt$ of $R$, see Definition \ref{wt}.
We fix a reflectable set
$$\Pi=\{\alpha_i\mid i\in I\}\subseteq R^\times$$
where $I$ is an index set. Recall that $p:A\longrightarrow A^0$ is the projection with respect to
the decomposition $A=\dot A\oplus A^0$ of $A$. Since $A$ is identified with the subgroup $1\otimes A$ of $\v$, one might consider $p$ as
the restriction of the projection $\v\longrightarrow \v^0$, with respect to the decomposition
$\v=\vd\oplus \v^0$.
For $\a\in R$, let $p_j(\a)\in\mathbb Q$ be the $j$-th coordinate of $p(\a)$ with respect to the basis $\bes^0$ of $\v^0$,
 namely
$$p(\a)=\sum_{j\in J}p_j(\a)\sigma_j,$$
where $p_j(\a)=0$, for all but a finite number of $j\in J$.

Now let $(\a_1,\ldots,\a_k)\in\alt(\Pi)$. Then

$$0=\sum_{i=1}^k(-1)^i\sgn(\a_i)p(\a_i)=\sum_{j\in J}\sum_{i=1}^k(-1)^i\sgn(\a_i)p_j(\a_i)\sg_j.$$ So for
$s\in J$, we have
\begin{eqnarray*}
0&=&\lam_s\big(\sum_{j\in J}\sum_{i=1}^k(-1)^i\sgn(\a_i)p_j(\a_i)\sg_j\big)\\
&=&\sum_{j\in J}\sum_{i=1}^k(-1)^i\sgn(\a_i)p_j(\a_i)\lam_s(\sg_j)\\
&=&\sum_{i=1}^k(-1)^i\sgn(\a_i)p_s(\a_i).
\end{eqnarray*}
Summarizing the above discussion, we have
\begin{equation}\label{eq5}
(\a_1,\ldots,\a_k)\in\alt(\Pi)\Longleftrightarrow\sum_{i=1}^k(-1)^i\sgn(\a_i)p_s(\a_i)=0
\hbox{ for all }s\in J.
\end{equation}

For our next result, we note that for each $j\in J$ and $w\in\wt$, we have $\lam_j-w(\lam_j)\in\v$.
\begin{lem}\label{W AR Lem}
Let $w=w_{\alpha_{1}}\cdots w_{\alpha_{k}}\in\wt$, $\a_i\in\Pi$. Then, for $j\in J$, we have
\begin{equation*}
\sgn(\lam_j-w(\lam_j))=\sum_{i=1}^k(-1)^{i}\sgn(\alpha_{i})p_j(\a_i)
\end{equation*}
and
\begin{equation*}
p(\lam_j-w(\lam_j))=\sum_{s=1}^kp_j(\a_s)p(\a_s)+2\sum_{s=2}^k(-1)^sp_j(\a_s)\sgn(\alpha_{s})
\sum_{r=1}^{s-1}(-1)^{r}\sgn(\alpha_{r})p(\a_r).
\end{equation*}
\end{lem}
\begin{proof}
We have $w_{\a_k}(\lam_j)=\lam_j-(\lam_j,\a_k)\a_k$ so
$$p(w_{\a_k}(\lam_j)-\lam_j)=(\lam_j,\a_k)p(\a_k)=\lam_j(p(\a_k))p(\a_k)=p_j(\a_k)p(\a_k).$$
Now if $w'=w_{\a_1}\cdots w_{\a_{k-1}}$, then
\begin{equation*}
w(\lam_j)=w'(\lambda_j-p_j(\a_k)\alpha_{k})=w'(\lam_j)-p_j(\a_k)w'(\a_k)
\end{equation*}
Then by induction steps and Lemma \ref{action lem}, we have
\begin{eqnarray*}
p(\lam_j-w(\lam_j))&=&p(\lam_j-w'(\lam_j))+p_j(\a_k)p(w'(\a_k))\\
&=&
\sum_{s=1}^{k-1}p_j(\a_s)p(\a_s)+2\sum_{s=2}^{k-1}(-1)^s\sgn(\alpha_{s})
\sum_{r=1}^{s-1}(-1)^{r}\sgn(\alpha_{r})p(\a_r)\\
&&+ p_j(\a_k)\big(p(\a_k)-2(-1)^{k-1}\sgn(\a_k)\sum_{r=1}^{k-1}(-1)^r\sgn(\a_r)p(\a_r)\big)\\
&=&\sum_{s=1}^{k}p_j(\a_s)p(\a_s)+2\sum_{s=2}^{k}(-1)^s\sgn(\alpha_{s})
\sum_{r=1}^{s-1}(-1)^{r}\sgn(\alpha_{r})p(\a_r).
\end{eqnarray*}
Also
\begin{eqnarray*}
\sgn(\lam_j-w(\lam_j))&=&\sgn(w'(\lam_j))-p_j(\a_k)\sgn(w'(\a_k))\\
&=&\sum_{i=1}^{k-1}(-1)^{i}\sgn(\alpha_{i})p_j(\a_i)-(-1)^{k-1}p_j(\a_k)\sgn(\a_k)\\
&=&\sum_{i=1}^k(-1)^i\sgn(\a_i)p_j(\a_i).
\end{eqnarray*}
\end{proof}

The following is now clear from Lemma \ref{W AR Lem}.
\begin{pro}\label{W AR}
Let $w=w_{\alpha_{1}}\cdots w_{\alpha_{k}}\in\wt$, $\a_i\in R^\times$. Then $w=1$ in $\wt$
if and only if $k$ is even and for all $j\in J$,
\begin{equation}\label{eq6}
\sum_{i=1}^k(-1)^{i}\sgn(\alpha_{i})p_j(\a_i)=0
\end{equation}
and
\begin{equation}\label{eq7}
\sum_{s=1}^kp_j(\a_s)p(\a_s)+2\sum_{s=2}^k(-1)^sp_j(\a_s)\sgn(\alpha_{s})
\sum_{r=1}^{s-1}(-1)^{r}\sgn(\alpha_{r})p(\a_r)=0.
\end{equation}
\end{pro}
\begin{proof}
Let $w$ be as in the statement.
From definition of $\vt$, it is clear that
$w=1$ in $\wt$ if and only if $w=1$ in $\w$ and $w(\lam_j)=\lam_j$ for all $j\in J$.
Now from Proposition \ref{cor1}, we have $w=1$ in $\w$ if and only if $k$ is even and (\ref{eq6}) is satisfied.
Also by Lemma \ref{W AR Lem}, $w(\lam_j)=\lam_j$ if and only if (\ref{eq6}) and (\ref{eq7}) hold.
\end{proof}

\begin{rem}\label{rem3}
Since $w_\a=w_{-\a}$, we may assume in Proposition \ref{W AR} that $\sgn(\a_i)=1$ for all $i$. Then
the statement can be written in the form; $w=1$ in $\wt$ if and only if $k$ is even and for all $j\in J$,
\begin{equation*}
\sum_{i=1}^k(-1)^{i}p_j(\a_i)=0
\end{equation*}
and
\begin{equation*}
\sum_{s=1}^kp_j(\a_s)p(\a_s)+2\sum_{s=2}^k(-1)^sp_j(\a_s)
\sum_{r=1}^{s-1}(-1)^{r}p(\a_r)=0.
\end{equation*}
\end{rem}

Suggested by Proposition \ref{W AR Lem}, we make the following definition.
\begin{DEF}\label{def4}
Let $P$ be a subset of $R^\times$. We call $(\a_1,\ldots,\a_k)$, $\a_i\in P$, a {\it hyper-alternating $k$-tuple} in $P$ if $k$ is even and
(\ref{eq6}) and (\ref{eq7}) hold, for all $j\in J$.
We denote by $\talt(P)$ the set of all hyper-alternating  $k$-tuples in $P$, for all $k$.
By Proposition \ref{W AR}, $w_{\a_1}\cdots w_{\a_k}=1$ in $\wt$ if and only if $(\a_1,\ldots,\a_k)\in\talt(R^\times)$.
\end{DEF}

Using an argument similar to the proof of Theorem \ref{W(V) AP}, we summarize the results of this section in the following theorem.

\begin{thm}\label{W AP}
Let $\Pi\subseteq R^\times$ be a reflectable set for $R$. Then, $\wt$ is isomorphic
to the presented group $G$, defined by
\begin{itemize}
\item[] generators: $y_{\alpha},\;\a\in\Pi$,
\item[] relations: $y_{\a_1}\cdots y_{\a_k}$, $(\a_1,\ldots,\a_k)\in\talt(\Pi)$.
\end{itemize}
\end{thm}

We note that if $w=w_{\a_1}\cdots w_{\a_k}$, $\a_i$'s in $R^\times$, then depending on the context we may consider $w$ either
as an element of $\w$ or $\wt$.

\begin{pro}\label{SM lem}
Let $R^0\not=\{0\}$ and $w:=w_{\a_1}\cdots w_{\a_k}$, $\a_i$'s in $R^\times$.
The following statements are equivalent.
\begin{itemize}
\item[(i)] $(\a_1,\ldots,\a_k)\in\alt(R^\times)$,
\item[(ii)] $w=1$ in $\w$ (or equivalently $w_{\mid_\v}=\mathrm{id}_\v$),
\item[(iii)] if $w=w_{\b_1}\cdots w_{\b_m}$, $\b_i$'s in $R^\times$, then $(\b_1,\ldots,\b_m)\in\alt(R^\times)$,
\item[(iv)] $w\in Z(\wt)$.
\end{itemize}
\end{pro}

\begin{proof}
The implications (i)$\Longleftrightarrow$(ii)$\Longleftrightarrow$(iii) follow at once from Proposition \ref{cor1}, or Theorem \ref{W(V) AP}.

(ii)$\Longrightarrow$(iv). Assume $w=1$ in $\w$. Therefore $w(\a)=\a$ for all $\a\in\v$ and so for any
$\a\in R^\times$ we have $ww_\alpha w^{-1}=w_{w(\a)}=w_\alpha.$
Thus $w\in Z(\widetilde{\w})$.

(iv)$\Longrightarrow$(ii). Let $z\in Z(\wt)$. For each $\alpha\in R^\times$, we have $zw_\alpha z^{-1}=w_\alpha$. By (\ref{Conj Pro}), we have $$w_{z(\alpha)}=w_\alpha.$$
From the definition of a reflection and the fact that $R$ is a reduced root system, we have $$z(\alpha)=\pm\alpha.$$
It is enough to show that $z$ acts as identity on $\epsilon+S$,
which is an spanning set of $\v$. We use the fact that each $w\in\widetilde{\w}$
acts as identity map on$\v^0$. Since $R^0\not=\{0\}$, there exists $0\not=\sg\in S$.
Then if $z(\epsilon+\sigma)=-\epsilon-\sigma$, we have $$z(\epsilon)+\sigma=-\epsilon-\sigma,$$
which is a contradiction in both cases $z(\epsilon)=\pm\epsilon$.
Thus $$z(\epsilon+\sigma)=\epsilon+\sigma.$$
Then
$$z(\epsilon)=z(\epsilon+\sigma-\sigma)=\epsilon+\sigma-\sigma=\epsilon.$$
\end{proof}

\begin{cor}\label{centext}
The group $\wt$ is a central extension of $\w$ by $Z(\wt)$, namely,
$$1\longrightarrow Z(\wt)\stackrel{i}{\hookrightarrow}\wt\stackrel{\varphi}{\longrightarrow}\w\longrightarrow 1$$
is a short exact sequence, where $\varphi(w)=w_{|_\v}$, for $w\in\wt$.
\end{cor}

\begin{proof}
As it was mentioned at the end of Section \ref{ARS Def}, the map $\varphi$ is an epimorphism. Now, by Proposition  \ref{SM lem}, we have $\mathrm{Ker}\varphi=Z(\wt)$.
\end{proof}

One knows that an interesting subclass of affine reflection
systems is the class for which the center $Z(\wt)$ of $\wt$ is a
free abelian group. The next theorem shows that, in this case, the
existence of a presentation for $\w$ is equivalent to the
existence of a presentation for $\wt$. One should notice the
crucial role of $Z(\wt)$ in the proof.

\begin{con}\label{convention}
Suppose that $H$ is a group and $\{h_\alpha\mid \alpha\in P\}$ is a fixed subset of  $H$.
For a $k$-tuple $f=(\alpha_1,\dots,\alpha_k)$ in $P$, $k$ a positive integer, we
set
$$f_H:=h_{\alpha_1}\cdots h_{\alpha_k}\in H.$$
\end{con}

Now let $\Pi$ be a reflectable set for $R$, and $\{z_l\}_{l\in L}$ be a fixed set of generators for $Z(\wt)$.
Then, using Proposition \ref{SM lem} together with the fact that $\{w_\alpha\;|\;\alpha\in\Pi\}$ generates $\wt$,
each $z_\ell$, $\ell\in L$, can be written in the form $$z_\ell=f_{\wt}$$ for some
$f=(\a_1,\ldots\a_k)\in\alt(\Pi)$, see Convention \ref{convention}.

The term $[x,y]$
denotes the commutator $xyx^{-1}y^{-1}$ of two elements $x,y$ of a group.

\begin{thm}\label{W W(V) CentP}
Let $\Pi$ be a reflectable set for $R$ and assume that $Z(\wt)$ is a free abelian group. Fix a free basis $\{z_l\}_{l\in L}$ for $Z(\wt)$,
and for each $\ell\in L$ let $f^\ell=(\a_1^\ell,\ldots,\a_{k_\ell}^\ell)$ be a fixed element of $\alt(\Pi)$ such that $z_\ell=f^\ell_{\wt}$.
Then the following statements are equivalent:

(i) The assignment $w_\a\mapsto x_\a$, $\a\in\Pi$, induces an isomorphism $\psi$ from the Weyl group $\w$ of $R$ onto the presented group $G$, defined by
\begin{itemize}
\item[] generators: $x_\alpha$, $\alpha\in\Pi$,
\item[] relations:
$x^2_\alpha$, $f^l_{G}$, $\alpha\in\Pi$, $l\in L$.
\end{itemize}

(ii) The assignment $w_\a\mapsto\xt_\a$, $\a\in\Pi$, induces an isomorphism $\theta$ from the hyperbolic Weyl group $\wt$ of $R$ onto the presented group $\G$, defined by
\begin{itemize}
\item[] generators: $\xt_\alpha$, $\alpha\in\Pi$,
\item[] relations:
$\xt_\alpha^2$, $[\xt_\alpha,f^l_{\G}]$, $\alpha\in\Pi$, $l\in L$.
\end{itemize}
\end{thm}

\begin{proof}
(i)$\Longrightarrow$(ii). Consider the assignment
\begin{eqnarray*}
\theta:\{\xt_\alpha\}_{\alpha\in\Pi}\longrightarrow\{w_\alpha\}_{\alpha\in\Pi}\\
\xt_{\a}\quad\longmapsto\;w_\a.\quad\quad\;
\end{eqnarray*}
Since any defining relation in $\widetilde{G}$ corresponds to a relation in $\wt$, the map $\theta$ can be extended to an epimorphism from $\widetilde{G}$ onto $\wt$. We proceed with the proof by showing that $\theta$ is injective.

Consider the subgroup $\widetilde{Z}:=\langle f^l_{\G}\;|\;l\in L\rangle$ of $\widetilde{G}$.  We show that $\widetilde{Z}\cong Z(\wt)$ and $\widetilde{G}/\widetilde{Z}\cong\w$. From the defining relations of $\G$, it is clear that $\widetilde{Z}$ is contained in $Z(\widetilde{G})$. Since $$\theta(f^l_{\G})=f^l_{\wt}=z_l$$
and $\{z_l\;|\;l\in L\}$ is basis of $Z(\wt)$, it follows easily that $$\theta_1:=\theta_{|_{\widetilde{Z}}}:\widetilde{Z}\longrightarrow Z(\wt)$$
is an isomorphism.

Next, from Section \ref{ARS Def}, recall that $\varphi:\wt\longrightarrow\w$ is the epimorphism which is defined by $\varphi(w)=w_{|_\v}$, with $\mathrm{Ker}\varphi=Z(\wt)$, and consider the epimorphism $\varphi\circ\theta:\widetilde{G}\longrightarrow\w$. We have
$\widetilde{Z}\subseteq\mathrm{Ker}(\varphi\circ\theta)$. So, there is an epimorphism $\theta_2:\widetilde{G}/\widetilde{Z}\longrightarrow\w$ such that $\theta_2(\widetilde{x}\widetilde{Z})=\varphi\circ\theta(\widetilde{x})$, for $\xt\in\widetilde{G}$. In particular, for $\alpha\in\Pi$, we have $\theta_2(\xt_\alpha\widetilde{Z})=w_\alpha$.

Let $\pi:\widetilde{G}\longrightarrow\widetilde{G}/\widetilde{Z}$ be the canonical map and $\imath:\widetilde{Z}\longrightarrow\widetilde{G}$ be the inclusion. Then we have the following commutative diagram of exact sequences:
\begin{equation*}
\begin{array}{lllllllll}
1&\longrightarrow& Z(\widetilde{\w})&\stackrel{i}{\hookrightarrow}&\widetilde{\w}&\stackrel{\varphi}{\longrightarrow}&\w&\longrightarrow&1\\
&&\uparrow\theta_1&&\uparrow\theta&&\uparrow\theta_2&&\\
1&\longrightarrow& \widetilde{Z}&\stackrel{\imath}{\hookrightarrow}&\widetilde{G}&\stackrel{\pi}{\longrightarrow}&\widetilde{G}/\widetilde{Z}&\longrightarrow&1.
\end{array}
\end{equation*}
Since $\theta_1$ is injective, if we show that $\theta_2$ is injective, then we have $\theta$ is injective. Thus $\wt\cong\widetilde{G}$.

To show that $\theta_2$ is injective we show that $\theta_2$ is invertable. For $\alpha\in\Pi$, we have $(\xt_\alpha\widetilde{Z})^2=\widetilde{Z}$ and, for $l\in L$, we have $f^l_{\widetilde{G}}\widetilde{Z}=\widetilde{Z}$. Thus any defining relation in $G$ corresponds to a relation in $\widetilde{G}/\widetilde{Z}$. So, there is an epimorphism $\kappa:G\longrightarrow \widetilde{G}/\widetilde{Z}$, where $\kappa(x_\alpha)=\xt_\alpha\widetilde{Z}$, for $\alpha\in\Pi$. Ultimately, we have an epimorphism $\kappa\circ\psi:\w\longrightarrow\widetilde{G}/\widetilde{Z}$ such that $\kappa\circ\psi(w_\alpha)=\xt_\alpha\widetilde{Z}$. Clearly $\kappa\circ\psi$ is the inverse of $\theta_2$.

(ii)$\Longrightarrow$(i). From Theorem \ref{W(V) AP}, we know that the assignment $w_\a\mapsto y_\a$, $\a\in\Pi$, induces an isomorphism from $\w$ onto the
group $G'$ defined by
\begin{itemize}
\item[] generators: $y_\alpha$, $\alpha\in\Pi$,
\item[] relations:
$y_{\alpha_1}\cdots y_{\alpha_k}$, $(\a_1,\ldots,\a_k)\in
\alt(\Pi)$.
\end{itemize}
It is clear that every relation in $\G$ is a relation in $G$. Thus we can extend the natural one-to-one correspondence  $\rho:\{\xt_\a\}_{\a\in\Pi}\longrightarrow\{x_\a\}_{\a\in\Pi}$ to an epimorphism $\rho:\G\longrightarrow G$, where $\rho(\xt_\a)=x_\a$. Thus $\rho\circ\theta$ is an epimorphism from $\wt$ onto $G$.
We show
that $G\cong G'$.

Let
$\phi:\{x_\alpha\}_{\alpha\in\Pi}\longrightarrow\{y_\alpha\}_{\alpha\in\Pi}$
be the natural one-to-one correspondence between the set of
generators of $G$ and the set of generators of $G'$. Since, for
$\alpha\in\Pi$, we have $(\alpha,\alpha)\in\alt(\Pi)$ and $f^l\in
\alt(\Pi)$, for $l\in L$, each relation in $G$ corresponds naturally to a defining relation in $G'$.
Thus $\phi$ can be extended to a group epimorphism
$\psi:G\longrightarrow G'$, where $\psi(x_\alpha)=y_\alpha$, for
$\alpha\in\Pi$.

Now, let $y_{\alpha_1}\cdots y_{\alpha_k}$ be a relation in $G'$.
Thus $(\alpha_1,\dots,\alpha_k)\in\alt(\Pi)$. By Proposition \ref{SM
lem}, we have $w:=w_{\alpha_1}\cdots w_{\alpha_k}\in Z(\wt)$. So
there are $l_1,\dots,\l_t\in L$ and $n_1,\dots,n_t\in\mathbb{Z}$ such that
$$w=(z_{l_1})^{n_1}\cdots (z_{l_t})^{n_t}=(f^{l_1}_{\wt})^{n_1}\cdots (f^{l_t}_{\wt})^{n_t}.$$
Thus
\begin{eqnarray*}
x_{\alpha_1}\cdots x_{\alpha_k}&=&\rho\circ\theta(w)\\
&=&\rho\circ\theta((f^{l_1}_{\wt})^{n_1}\cdots (f^{l_t}_{\wt})^{n_t})\\
&=&(f^{l_1}_{G})^{n_1}\cdots (f^{l_t}_{G})^{n_t}\\
&=&1.
\end{eqnarray*}
So, any defining relation in $G'$ corresponds to a
relation in $G$. Thus $\phi^{-1}$ can be extended to a
group epimorphism $\eta:G'\longrightarrow G$. Since $\psi$ is an extension of $\phi$ and
$\eta$ is an extensions of $\phi^{-1}$, they are inverse of each other.
Thus $G\cong G'\cong\w$.
\end{proof}

\begin{rem}
(i) The proof of (i)$\Longrightarrow$(ii) of Theorem \ref{W W(V) CentP} is in fact a consequence of a more general result concerning presented groups, see \cite[Theorem 10.2]{JPOG}. However, to apply this general fact to our situation, it needs quite a bit of adjustments. For this reason, we preferred to give a direct proof for this special case.

(ii) According to \cite[Lemma 3.18 (i) and Corollary 3.29]{AEAWG}, if $R$ is an extended affine root system in the sense of \cite[Definition II.2.1]{AABGP}, then $Z(\wt)$ is a free abelian group of finite rank. So Theorem \ref{W W(V) CentP} is applicable to extended affine Weyl groups of type $A_1$. In the next section we will show that when $R$ is an extended affine root system, $\w$ is isomorphic to the group $G$ defined in Theorem \ref{W W(V) CentP}, for a special reflectable set $\Pi$ and a particular set of $f^l$'s.
\end{rem}

\section{\bf Application to extended affine Weyl groups}\setcounter{equation}{0}\label{RedPre}
Let $(A,\fm, R)$ be an affine reflection system.
In this section, we assume that $A^0$ is a free abelian group of rank $\nu$.
Then $1\otimes R$ as a subset of $\v:=\mathbb{R}\otimes_\mathbb{Q}A$ turns out to be an \textit{extended
affine root system} of type $A_1$ in the sense of \cite[Definition II.2.1]{AABGP}. So the hyperbolic Weyl group $\wt$ is just an extended affine Weyl group in the sense \cite[Definition 2.14]{AEAWG}. Now, similar to what we have seen in Section \ref{hweyl}, $R$ can be identified with $1\otimes R$ and
we have $\v=\vd\oplus\v^0$, where
$\vd=\spani_{\mathbb R}\ad$ and $\v^0=\spani_{\mathbb R}A^0$. Set
$$\Lam:=A^0.$$
Then
$$R=(S+S)\cup(\rd+S),$$
where, in this case, the pointed reflection space $S$ is a \textit{semilattice} in $\Lam$ in the sense of
\cite[Definition II.1.2]{AABGP}, namely $S$ is a subset of $\Lam$ satisfying
$$0\in S,\quad S\pm 2S\sub S,\quad \la S\ra=\Lam.$$

By \cite[Remark II.1.6]{AABGP}, we have $$S=\bigcup_{i=0}^m(\tau_i+2\Lambda),$$
where $\tau_i$'s represent distinct cosets of $2\Lambda$ in $\Lam$, for $1\leq i\leq m$, and $\tau_0=0$.  By \cite[Proposition II.1.11]{AABGP}, $\Lambda$ has a $\mathbb{Z}$-basis consisting of elements of $S$. So we may assume that
 $$\{\sg_1:=\tau_1,\ldots,\sg_\nu:=\tau_\nu\}$$
 is a $\mathbb{Z}$-basis of $\Lambda$. It follows from \cite[Theorem 3.1]{AYY} that
\begin{equation}\label{frb}
\Pi:=\{\ep,\ep+\tau_1,\ldots,\ep+\tau_m\}
\end{equation}
is a reflectable base for $R$. Considering these facts, there are two extreme
cases for $S$ which we would like to treat separately. The first case is when
$m=\nu$.
We call the corresponding root system the {\it baby} extended affine root system and we denote it by $R_b$. Another extreme case
is when $S$ is a lattice, namely $S=\Lam$. The corresponding root system is called the {\it toroidal}  root system which we denote it by $R_t$.  With our conventions, for any extended affine root system $R$ of type $A_1$ in $A$, we have $R_b\subseteq R\subseteq R_t$. This justifies special treatment of $R_b$ and $R_t$.

The following proposition is essential for obtaining a new presentation for $\w$.

\begin{pro}\label{baby W}
We have
\begin{itemize}
\item[(i)] $w_{\alpha+\sigma+\delta}=w_{\alpha+\sigma}w_\alpha
w_{\alpha+\delta}$, for $\alpha\in R_t^\times$ and $\sigma,\delta\in
\Lam$.
\item[(ii)]
$w_{\alpha+k\sigma}w_\alpha=(w_{\alpha+\sigma}w_\alpha)^k$, $k\in\mathbb{Z}$, $\alpha\in R_t^\times$ and $\sigma\in\Lam$.
\item[(iii)]
We have $\w=\w_b$, for any extended affine root system $R$ in $A$.
\end{itemize}
\end{pro}

\begin{proof}
(i)-(ii) The tuples
\begin{equation*}
(\alpha+\sigma+\delta,\alpha+\sigma,\alpha,\alpha+\delta)\quad\quad\quad
\end{equation*}
\begin{equation*}
(\alpha,\alpha+k\sigma,\overbrace{\alpha+\sigma,\alpha,\dots,\alpha+\sigma,\alpha}^{2k}),\quad
for\;k>0
\end{equation*}
\begin{equation*}
(\alpha,\alpha+k\sigma,\overbrace{\alpha,\alpha+\sigma,\dots,\alpha,\alpha+\sigma}^{-2k}),\quad
for\;k<0\\
\end{equation*}
are elements of $\alt(R_t^\times)$. Also (ii) is nothing but $w_\alpha^2=1$, for $k=0$. Thus by Proposition \ref{cor1} and Definition \ref{semi-relation} (i) and (ii) hold.

(iii) It is enough to show that $w_\alpha\in\w_b$, for each $\alpha\in\pm\epsilon+\Lambda$. Since $w_\alpha=w_{-\alpha}$, we assume that $sgn(\alpha)=1$.
Let $\alpha=\epsilon+\sigma$ and $$\sigma=\sum_{i=1}^\nu
k_i\sigma_i,$$ where $k_i\in\mathbb{Z}$. Using (i), we have
\begin{equation*}
w_\alpha=w_{\epsilon+k_1\sigma_1}w_\epsilon\cdots
w_{\epsilon+k_{\nu-1}\sigma_{\nu-1}}w_\epsilon
w_{\epsilon+k_\nu\sigma_\nu}.
\end{equation*}
Now for each $i$, from (ii) we have $$w_{\ep+k_i\sg_i}w_\ep=(w_{\ep+\sg_i}w_\ep)^{k_i}.$$ This way we obtain an expression of $w_\alpha$ with respect to the reflections based on elements of $\Pi_0=\{\epsilon,\epsilon+\sigma_1,\dots,\epsilon+\sigma_\nu\}$. Since $\Pi_0$ is a subset of $R_b$, we have $w_\alpha\in\w_b$.
\end{proof}

Here we offer a finite presentation for $\w$ which is essential for the rest of this section. Let $\Pi_0:=\{\a_0:=\epsilon,\a_1:=\epsilon+\sigma_1,\dots,\a_\nu:=\epsilon+\sigma_\nu\}$. First we analyze the elements of $\alt(\Pi_0)$ which we need for our presentation.

\begin{rem}\label{eaalt} Let $f=(\a_{j_1},\dots,\a_{j_k})$ be an element of
$\alt(\Pi_0)$. The followings can be easily checked from
definition of $\alt(\Pi_0)$.

(i) From Corollary \ref{action cor}, for any $1\leq s\leq k$, we have
\begin{equation*}
(\a_{j_1},\dots,\a_{j_{s-1}},\a_{j_{s+2}},\a_{j_{s+1}},\a_{j_s},\a_{j_{s+3}},\dots,\a_{j_k})\in\alt(\Pi_0).
\end{equation*}

(ii) Note that $\Pi_0$ is a linearly independent set. So, if $n_i$ is the number of $\a_i$ in $f$, for $0\leq i\leq\nu$, namely $$n_i=|\{s\;|\;1\leq s\leq k,\;j_s=i\}|,$$
then $n_i$ is even. Also, for each $1\leq r\leq k$ there is an odd  integer $p$ such that $j_r=j_{r+p}$.

(iii) If $k=2$ then we have $f=(\a_i,\a_i)$, for some $0\leq
i\leq\nu$. If $k=4$ then either $f=(\a_i,\a_i,\a_j,\a_j)$ or
$f=(\a_i,\a_j,\a_j,\a_i)$, for some $0\leq i,j\leq\nu$. The only
possible alternating 6-tuple $f=(\a_{j_1},\dots,\a_{j_6})$ which
does not contain an alternating 4-tuple of the form
$(\a_{j_s},\dots,\a_{j_{s+3}})$ has to be of the  form
$f=(\a_i,\a_j,\a_m,\a_i,\a_j,\a_m)$, for some $0\leq
i,j,m\leq\nu$.
\end{rem}

\begin{thm}\label{Baby W pre}
Let $R$ be an extended affine root system of type $A_1$ and nullity $\nu$, then the Weyl group $\w$ of $R$ is isomorphic to the group $G$, defined by
\begin{itemize}
\item[] generators: $x_i$, $0\leq i\leq\nu$,
\item[] relations:
$x^2_k$, $(x_0x_ix_j)^2$, $0\leq k\leq\nu$, $1\leq i<j\leq\nu$.
\end{itemize}
\end{thm}

\begin{proof}
By Proposition \ref{baby W}, we have $\w=\w_b$. Since $\Pi_0$ is a reflectable base for $R_b$, using Theorem \ref{W(V) AP}, $\w_b$ is isomorphic to the group $G'$, defined by
\begin{itemize}
\item[] generators: $y_\alpha$, $\alpha\in\Pi_0$,
\item[] relations: $y_{\alpha_1}\cdots y_{\alpha_k}$, $(\a_1,\ldots,\a_k)\in\alt(\Pi_0)$.
\end{itemize}
Let $\phi:\{x_i\}_{i=0}^\nu\longrightarrow\{y_{\alpha_i}\}_{i=0}^\nu$
be the natural one-to-one correspondence between the set of
generators of $G$ and the set of generators of $G'$. Since, for
$\alpha\in\Pi_0$, we have $(\alpha,\alpha)\in\alt(\Pi_0)$ and $(\a_0,\a_i,\a_j,\a_0,\a_i,\a_j)\in\alt(\Pi_0)$, for $0\leq i<j\leq\nu$, each relation in $G$ corresponds naturally to a defining relation in $G'$.
Thus $\phi$ can be extended to a group epimorphism
$\psi:G\longrightarrow G'$, where $\psi(x_i)=y_{\alpha_i}$, for
$0\leq i\leq \nu$. We show that $\psi$ is an isomorphism.

Let $y_{\a{j_1}}\cdots y_{\a_{j_k}}$ be a defining relation in $G'$. Then $(\a_{j_1},\cdots,\a_{j_k})\in\alt(\Pi_0)$. We prove that $x:=x_{j_1}\cdots x_{j_k}$ is a relation in $G$. Notice that $k$ is even. We show this by induction on $m$, where $k=2m$.

From Remark \ref{eaalt} (iii), for $k=2,4$, we have $x$ is one of the expressions
$x_i^2$, $x_i^2x_j^2$ or $x_ix_j^2x_i$ which are clearly relations in $G$.

Next we examine the case $k=6$. It follows easily from defining relations of $G$ that
\begin{equation}\label{zp2}
\begin{array}{c}
(x_rx_sx_t)^2=1\quad\mathrm{and}\quad x_rx_sx_t=x_tx_sx_r,\vspace{2mm}\\
\hbox{if at least one of $r,s$ or $t$ is 0}.
\end{array}
\end{equation}
If
$(\a_{j_1},\dots,\a_{j_6})$ contains an alternating 4-tuple
$f'=(\a_{j_s},\ldots,\a_{j_{s+3}})$ and $y$ is the element in
$G$ corresponding to $f'$, then $x$ has to have one of the
forms $x_j^2y$, $x_jyx_j$ or $yx_j^2$ and so by the case $k=4$,
$x$ is a relation in $G$. Thus by Remark \ref{eaalt}, we may
assume that the alternating 6-tuple under consideration is
$(\a_r,\a_s,\a_t,\a_r,\a_s,\a_t)$, where none of $r,s$ and $t$
is 0. By (\ref{zp2}), we have
\begin{eqnarray*}
x&=&(x_rx_sx_t)^2\\
&=&x_0^2(x_rx_sx_t)^2\\
&=&x_0x_sx_rx_0x_tx_rx_sx_t\\
&=&x_0x_sx_tx_tx_rx_0x_tx_rx_0x_0x_sx_t\\
&=&x_0x_sx_t(x_tx_rx_0)^2x_0x_sx_t\\
&=&x_0x_sx_tx_0x_sx_t=1.
\end{eqnarray*}
Thus $x$ is a relation in $G$ when $k=6$, and
\begin{equation}\label{p2}
x_rx_sx_t=x_tx_rx_s,\quad\hbox{ for all }0\leq r,s,t\leq\nu.
\end{equation}

Now, let $m>3$ (or $k>6$) and assume every expression in $G$ corresponding to an alternating
$2n$-tuple is a relation in $G$, for $n<m$. First
assume that for some $1\leq r\leq k-1$, $j_r=j_{r+1}$.  Since
$x_{j_r}^2=1$, we have
$$x=x_{j_1}\cdots x_{j_k}=x_{j_1}\cdots x_{j_{r-1}} x_{j_{r+2}} \cdots x_{j_k}$$
and $(\a_{j_1},\dots,\a_{j_{r-1}},\a_{j_{r+2}},\dots,\a_{j_k})$
is an alternating $(k-2)$-tuple. Thus $x$ is a relation in
$G$, by induction steps. So, we may assume that
$j_r\not=j_{r+1}$, for all $1\leq r\leq k-1$. From Remark
\ref{eaalt} (ii), the number of $\a_{j_1}$ appearing in
$(\a_{j_1},\dots,\a_{j_k})$ is even and there is an even
integer $2\leq s\leq k$ such that $j_1=j_s$. Since by
(\ref{p2}) $x_{j_1}x_{j_2}x_{j_3}=x_{j_3}x_{j_2}x_{j_1}$, we
have
$$x=x_{j_1}\cdots x_{j_k}=x_{j_3}x_{j_2}x_{j_1}x_{j_4}\cdots x_{j_k}.$$
By repeating this process we can move $x_{j_1}$ next to $x_{j_s}$, namely
$$x=x_{j_1}\cdots x_{j_k}=x_{j_3}x_{j_2}x_{j_5}x_{j_4}\cdots x_{j_{s-1}}x_{j_{s-2}}x_{j_1}x_{j_s}\cdots x_{j_k}.$$
By Remark \ref{eaalt} (i), $(\a_{j_3},\a_{j_2},\a_{j_5},\a_{j_4}\dots \a_{j_{s-1}},\a_{j_{s-2}},\a_{j_1},\a_{j_s},\dots,\a_{j_k})$ is an alternating  $k$-tuple. Thus $(\a_{j_3},\a_{j_2},\a_{j_5},\a_{j_4}\dots \a_{j_{s-1}},\a_{j_{s-2}},\a_{j_{s+1}},\dots,\a_{j_k})$ is an alternating  $(k-2)$-tuple and
$$x=x_{j_1}\cdots x_{j_k}=x_{j_3}x_{j_2}x_{j_5}x_{j_4}\cdots x_{j_{s-1}}x_{j_{s-2}}x_{j_{s+1}}\cdots x_{j_k}.$$
By induction steps, the right hand side is a relation in $G$. Thus every relation in $G'$ corresponds to a relation in $G$. So, $\phi^{-1}$ can be extended to a group epimorphism $\eta:G'\longrightarrow G$, satisfying $\psi\circ\eta=1$ and $\eta\circ\psi=1$. Thus $G\cong G'\cong\w$.
\end{proof}

A slight modification of generators and relations in the statement of Theorem \ref{Baby W pre} provides a new presentation for
$\w$ which is more useful for our purposes. We do this in the following proposition. Even though the proof is elementary and straightforward,
we provide the details for the convenience of the reader.

\begin{pro}\label{Baby W spre}
Let $B$ be a subset of $\{(i,j)\;|0\leq i<j\leq\nu\}$. Then $\w$ is isomorphic to the group $G$, defined by
\begin{itemize}
\item[] generators: $x_k$, $x_{(i,j)}$, $0\leq k\leq\nu$, $(i,j)\in B$,
\item[] relations:
$x^2_k$, $x_{(i,j)}^2$, $0\leq k\leq\nu$, $(i,j)\in B$;  $x_{(i,j)}x_ix_0x_j$, $(i,j)\in B$; $(x_ix_0x_j)^2$, $(i,j)\not\in B$.
\end{itemize}
\end{pro}

\begin{proof}
If $B=\emptyset$ this is the same as Theorem \ref{Baby W pre}.
Suppose that $B\not=\emptyset$. By Theorem \ref{Baby W pre}, $\w$ is isomorphic to the group $G'$, defined by
\begin{itemize}
\item[] generators: $y_i$, $0\leq i\leq \nu$,
\item[] relations:
$y^2_k$, $(y_0y_iy_j)^2$, $0\leq k\leq\nu$, $1\leq i<j\leq\nu$.
\end{itemize}
By the proofs of Theorems \ref{W(V) AP} and \ref{Baby W pre}, this isomorphism is in fact induced by the assignment $w_{\ep+\sg_i}\mapsto y_i$ for $0\leq i\leq\nu$.
We show that $G\cong G'$. For $(i,j)\in B$, we set $y_{(i,j)}:=y_jy_0y_i\in G'$. Let $Y=\{y_k\}_{k=0}^\nu\cup\{y_{(i,j)}\}_{(i,j)\in B}$ and define $\phi:Y\longrightarrow\{x_k\}_{k=0}^\nu\cup\{x_{(i,j)}\}_{(i,j)\in B}$  by $\phi(y_k)=x_k$, for $0\leq k\leq\nu$, and $\phi(y_{(i,j)})=x_{(i,j)}$, for $(i,j)\in B$. It is obvious that $Y$ is a set of generators for $G'$. In G, we have $x_k^2=1$, for $1\leq k\leq\nu$. Since $(x_ix_0x_j)^2=1$, $(i,j)\not\in B$, we have $(x_jx_ix_0)^2=x_j^2=1$. Thus $$(x_0x_ix_j)^2=(x_jx_ix_0)^{-2}=1,$$
for $(i,j)\not\in B$.
Also, $(x_jx_0x_i)^2=x_{(i,j)}^2=1$, $(i,j)\in B$. Thus $$(x_0x_ix_j)^2=x_j(x_jx_0x_i)^2x_j=x_j^2=1,$$ for $(i,j)\in B$. Thus $\phi$ can be extended to an epimorphism $\psi:G'\longrightarrow G$.

On the other hand, in $G'$, we have
$y_{(i,j)}y_iy_0y_j=y_jy_0y_iy_iy_0y_j=1$ and $y_{(i,j)}^2=1$,
for $(i,j)\in B$, $$(y_iy_0y_j)^2=y_j(y_jy_iy_0)^2y_j=y_j(y_0y_iy_j)^{-2}y_j=y_j^2=1,$$ for $(i,j)\not\in B$
and $y_k^2=1$, $0\leq k\leq\nu$. So, any defining relation in $G$ corresponds to a relation in $G'$. Thus $\phi^{-1}$ can be extended to a group homomorphism $\eta:G\longrightarrow G'$. Since $\psi$ is an extension of $\phi$ and $\eta$ is an extension of $\phi^{-1}$, they are inverse of each other and $G\cong G'\cong\w$. This completes the proof. Note that under this isomorphism the generator $x_k$ of $G$ maps to the element $w_{\ep+\sg_k}$ of $\w$, $0\leq k\leq\nu$, and the generator $x_{(i,j)}$ of $G$ maps to the element $w_{\ep+\sg_j}w_\ep w_{\ep+\sg_i}$ of $\w$, $(i,j)\in B$.
\end{proof}

Recall that each $\a\in R^\times$ can be written uniquely in the form $\pm\ep+\sum_{i=1}^\nu s_i\sg_i$ mod $2\Lam$, where $s_i\in\{0,1\}$ for all $i$.
Let $\Pi=\{\a_0,\ldots,\a_m\}$ be as in (\ref{frb}). Set $\supp(\a)=\{i\mid s_i\not=0\}$, and
$$\supp(\Pi)=\{\supp(\a)\mid\a\in\Pi\}.$$
Since $\Pi$ is a reflectable base, $\tau_i$'s represent distinct cosets of $2\Lam$ in $\Lam$, so $\tau_i=\tau_j$ if and only if
$\supp(\a_i)=\supp(\a_j)$.
Here we consider all sets $\supp(\a)$ as ordered sets, namely if $\supp(\a)=\{i_1,\ldots,i_t\}$, then $i_1 <i_2 <\cdots <i_t$.
We call the reflectable base $\Pi$ {\it elliptic-like} if $|\supp(\a)|\in\{0,1,2\}$ for all $\a\in\Pi$, or equivalently
if $|\supp(\a_k)|=2$ for $\nu+1\leq k\leq m$. Since $|\supp(\a)|\leq\nu$ for $\a\in R^\times$, all extended affine root systems of nullity $\leq 2$ are elliptic-like.
Finally we set
$$B_\Pi=\{(i,j)\mid \{i,j\}\in\supp(\Pi),\;1\leq i<j\leq\nu\}.$$
For $(i,j)\in B_\Pi$, we denote by $\a_{i_j}$ the unique element in $\Pi$ with $\supp(\a)=\{i,j\}$.


\begin{pro}\label{Baby Tor Pre}
Let $R$ be an extended affine root system in $A$ and $\Pi$ be the reflectable base for $R$ as in (\ref{frb}). Assume that $\Pi$ is elliptic-like.
Then $\wt$ is isomorphic to the group $\G$, defined by
\begin{itemize}
\item[] generators: $\xt_k$, $0\leq k\leq m$,
\item[] relations:
$\xt^2_k$, $[\xt_k,\xt_s\xt_i\xt_0\xt_j]$, if $\{i,j\}=\supp(\a_s)$; $[\xt_k,(\xt_i\xt_0\xt_j)^2]$, if $\{i,j\}\not\in\supp(\Pi)$, $0\leq k\leq m$, $1\leq i<j\leq\nu$, $\nu+1\leq s\leq m$.
\end{itemize}
\end{pro}

\begin{proof}
We proceed with the proof in the following steps.

Step1. We show that $Z(\wt)$ is a free abelian group with
basis $\{z_{ij}\;|\;1\leq i<j\leq\nu\}$, where $z_{ij}$'s are defined as follows.
If $\{i,j\}\in\supp(\Pi)$ set $z_{ij}:=w_{\epsilon+\tau_{i_j}}w_{\epsilon+\tau_i}w_{\epsilon}w_{\epsilon
+\tau_j}$, where $i_j$ is the unique integer satisfying $\supp(\a_{i_j})=\{i,j\}.$ If $\{i,j\}\not\in\supp(\Pi)$ set
$z_{ij}:=(w_{\epsilon+\tau_i}w_{\epsilon}w_{\epsilon+\tau_j})^2$. From \cite[Lemma 3.18(i) and Corollary
3.29]{AEAWG}, we know that the center  $Z(\wt)$ of $\wt$ is a free
abelian group of rank $\nu(\nu-1)/2$. For $1\leq i<j\leq \nu$ define $c_{ij}\in\hbox{GL}(\tilde{\v})$ by
$$c_{ij}(v)=v,\quad\hbox{and}\quad c_{ij}\lambda_k=\lambda_k-\delta_{kj}\sigma_i+\delta_{ki}\sigma_j,\quad(v\in\v,\;1\leq k\leq\nu).$$
By \cite[Proposition 2.2
(vi)]{ASFP},
$$\{c_{ij}|(i,j)\in
B_{\Pi}\}\cup\{c_{ij}^2|(i,j)\not\in
B_{\Pi}\}$$ is a free basis for the group $Z(\wt)$.
We are done if we show that $z_{ij}=c_{ij}$ for $(i,j)\in B_\Pi$ and $z_{ij}=c_{ij}^2$ for $(i,j)\not\in B_\Pi$.
Now for $(i,j)\in B_{\Pi}$, we have
$\epsilon+\sigma_i+\sigma_j\in R^\times$, and so $z_{ij}\in
\widetilde{\w}$. Moreover, by Lemma \ref{W AR
Lem} (or a simple verification),
$$z_{ij|_\v}=\mathrm{id}_\v$$ and
$$z_{ij}\lambda_k=\lambda_k-\delta_{kj}\sigma_i+\delta_{ki}\sigma_j=c_{ij}\lambda_k,$$
for $1\leq k\leq\nu$. Thus $z_{ij}=c_{ij}$. For $(i,j)\not\in B_{\Pi}$, we have
$$z_{ij|_\v}=\mathrm{id}_\v$$ and
$$z_{ij}\lambda_k=\lambda_k-2\delta_{kj}\sigma_i+2\delta_{ki}\sigma_j=c_{ij}^2\lambda_k.$$
Thus $z_{ij}=c_{ij}^2$.

Step 2. We show that the assignment $w_{\a_k}\mapsto x_{\a_k}$,
$0\leq k\leq m$, induces an isomorphism from the Weyl group $\w$
onto the group $G$ defined by
\begin{itemize}
\item[] generators: $x_{\a_k}$, $0\leq k\leq m$, \item[]
relations: $x^2_{\a_k}$, $0\leq k\leq m$, $f_G^{i_j}$, $1\leq
i<j\leq\nu$, where  $f^{i_j}:=(\a_{i_j},\a_i,\a_0,\a_j)$ for $(i,j)\in B_\Pi$,  and $f^{i_j}:=(\a_i,\a_0,\a_j,\a_i,\a_0,\a_j)$ if
$(i,j)\not\in B_\Pi$.
\end{itemize}
(Recall that for $(i,j)\in B_\Pi$, $\a_{i_j}$ is the unique element in $\Pi$ with $\supp(\a_{i_j})=\{i,j\}$.)
By Proposition \ref{Baby W spre}, the assignment $w_{\a_k}\mapsto
x_{k}$, $0\leq k\leq m$, induces an isomorphism from $\w$ onto the
group $G'$ defined by
\begin{itemize}
\item[] generators: $x_k$, $x_{(i,j)}$, $0\leq k\leq\nu$,
$(i,j)\in B_{\Pi}$,
\item[] relations:
$x^2_k$, $x_{(i,j)}^2$, $0\leq k\leq\nu$, $(i,j)\in B_{\Pi}$,
$x_{(i,j)}x_ix_0x_j$, $(i,j)\in B_{\Pi}$, $(x_ix_0x_j)^2$,
$(i,j)\not\in B_{\Pi}$.
\end{itemize}
Using the correspondence $x_{\a_i}\leftrightarrow x_i$ for $0\leq i\leq\nu$ and $x_{\a_k}\leftrightarrow x_{(i,j)}$ for $\nu+1\leq k\leq m$ with
$\supp(\a_k)=\{i,j\}$, the defining generators and relations of the groups $G$ and $G'$ coincide and so we identify them.

Step 3. By Step 1, Step 2 and Theorem \ref{W W(V) CentP}, $\wt$ is isomorphic to the group $\G$ defined by
\begin{itemize}
\item[] generators: $\xt_{\a_k}$, $0\leq k\leq m$, \item[]
relations: $\xt^2_{\a_k}$, $0\leq k\leq m$,
$[\xt_{\a_k},f^{ij}_{\G}]$, $1\leq i<j\leq\nu$, where
$f^{ij}=(\a_{i_j},\a_i,\a_0,\a_j)$ if $\{i,j\}\in\supp(\Pi)$,
and $f^{ij}=(\a_i,\a_0,\a_j,\a_i,\a_0,\a_j)$ if $\{i,j\}\not\in
\supp(\Pi)$.
\end{itemize}
Now using the correspondence $x_{\a_i}\leftrightarrow x_i$ for $0\leq i\leq\nu$ and $x_{\a_{i_j}}\leftrightarrow x_s$ for $\nu+1\leq s\leq m$ with
$\supp(\a_s)=\{i,j\}$, we have $\wt$ is isomorphic to the group defined by
\begin{itemize}
\item[] generators: $\xt_k$, $0\leq k\leq m$,
\item[] relations:
$\xt^2_k$, $[\xt_k,\xt_s\xt_i\xt_0\xt_j]$, if $\{i,j\}=\supp(\a_s)$; $[\xt_k,(\xt_i\xt_0\xt_j)^2]$, if $\{i,j\}\not\in\supp(\Pi)$, $0\leq k\leq m$, $1\leq i<j\leq\nu$, $\nu+1\leq s\leq m$.
\end{itemize}
\end{proof}

\begin{rem}\label{saito}
(i) We recall that if $\nu=0$  ($\nu=1$), then $R$ is a finite (affine) root system of type $A_1$.
Therefore Proposition \ref{Baby W spre}, together with Proposition \ref{Baby Tor Pre}, reproduces the following known presentations for $\w$ and $\wt$:
$$\w\cong(x_0\;|\;x_0^2)\cong\wt\cong \mathbb{Z}_2,\quad(\nu=0),$$
$$\w\cong(x_0,x_1\;|\;x_0^2,x_1^2)\cong\wt\cong \mathbb{Z}_2*\mathbb{Z}_2,\quad (\nu=1).$$

(ii) Proposition \ref{Baby Tor Pre} provides a finite presentation for the baby and the toroidal extended affine Weyl groups of type $A_1$ of nullity 2.
In this case, these are the only possible extended affine Weyl groups. Considering the nature of the relations one might consider
this presentation as a generalized Coxeter presentation. We encourage the interested reader to compare our defining set of generators and relations with those given in \cite{ST} for types $A_1^{(1,1)}$ and $A_1^{(1,1)^*}$.
\end{rem}

\section{\bf Appendix: A geometric approach}\setcounter{equation}{0}\label{geometric}
In this section, we provide a geometric approach to the proof of Theorem \ref{Baby W pre}. For
$w\in\w$ suppose that $w_{\a_1}\cdots w_{\a_k}$ is an
expression of $w$ with respect to $R^\times$. We define
\begin{equation}\label{Tfun}
\vep(w):=(-1)^k\andd
T(w):=\sum_{i=1}^k(-1)^{k-i}\sgn(\a_i)p(\a_i).
\end{equation}
From Lemma \ref{action lem}, it follows that the maps
$\vep:\w\longrightarrow\{-1,1\}$ and $T:\w\longrightarrow\Lam$
are well-defined, that is their definitions are independent of a particular
choice of given expressions for an element of $\w$.
Furthermore, $\vep$ is a group homomorphism and
\begin{equation}\label{newtemp1}
T(w_1w_2)=\vep(w_2)T(w_1)+T(w_2),\quad(w_1,w_2\in\w).\end{equation}
One can easily see that for $\a\in R$,
\begin{equation}\label{Raction}
w(\a)=\vep(w)\sgn(\a)\ep+p(\a)-2\sgn(\a)T(w),
\end{equation}
and so
$w$ is uniquely determined by $\vep(w)$ and $T(w)$.

For $\sg=\sum_{i=1}^\nu k_i\sg_i\in\Lam$ and $\eta\in\{\pm1\}$,
consider the $\nu$-simplex
$$B_{\sg,\eta}:=\{\sum_{i=1}^\nu(2k_i+\eta t_i)\sg_i\;|\;t_i\geq0,\;0\leq t_1+\cdots+t_\nu\leq 1\}.$$
Let $\bc=\{B_{\sg,\eta}\mid\sg\in\Lam,\;\eta\in\{\pm1\}\}$. Now
from (\ref{newtemp1}) and the fact that  $\vep$ is a group
homomorphism, it follows that $\w$ acts on $\bc$ by
\begin{equation}\label{GeoAction}
w\cdot B_{\sg,\eta}=B_{\sg+\eta T(w),\vep(w)\eta}.
\end{equation}
If $w\cdot B_{\sg,\eta}=B_{\sg,\eta}$, then $T(w)=0$ and $\vep(w)=1$, so $w=1$. This shows that $\w$ acts on $\bc$ freely,
in particular the action is faithful.
Moreover, (\ref{GeoAction})
shows that $\bc=\w\cdot B_{0,1}$, namely the action is transitive.

\begin{rem}
In parts (i)-(iii) below, we explain our intrinsic motivation for the action of $\w$ on $\bc$ defined in (\ref{GeoAction}) (a similar idea is given in \cite[\S 6.6]{Ka}).

(i) We show how the action of $\w$ on $R$ can be transferred to the
action (\ref{GeoAction}) on $\bc$. From (\ref{Raction}), it is easy
to see that the action of $\w$ on $\pm\ep+2\Lam$ is faithful. Using
the map $\eta\ep+2\sg\mapsto B_{\sg,\eta}$, $\eta\in\{\pm1\}$,
$\sg\in\Lam$, we can identify the set $\pm\ep+2\Lam$ with $\bc$. We
use this identification to transfer the action of $\w$ to an action,
denoted by $\bullet$, on $\bc$. In fact, since for $w\in\w$ we have
$$w(\eta\ep+2\sg)=\vep(w)\eta\ep+2(\sg-\eta T(w)),$$
then the corresponding action on $\bc$ reads as
$$w\bullet B_{\sg,\eta}=B_{\sg-\eta T(w),\vep(w)\eta}.$$
The action of $\w$ on $\pm\ep+2\Lam$, and so on $\bc$, is faithful
and transitive. The two actions $\bullet$ and $\cdot$ on $\bc$ are
related as follows:
$$w_\ep w w_\ep \bullet B_{\sg,\eta}=w_\ep w\bullet B_{\sg,-\eta}=w_\ep\bullet B_{\sg+\eta T(w),-\vep(w)\eta}=B_{\sg+\eta T(w),\vep(w)\eta}=w\cdot B_{\sg,\eta}.$$

(ii) Since we have identified
$\ep+2\sg$ and $-\ep+2\sg$ with $B_{\sg,1}$ and $B_{\sg,-1}$,
respectively, one can interpret this identification as a
polarization of elements of $2\Lam$ through elements of
$R^\times$, i.e., we consider  $B_{\sg,1}$ and $B_{\sg,-1}$ as
positive and negative poles for each element $2\sg\in2\Lam$,
respectively. In this way, the action of $\w$ on $\bc$ can be
interpreted on $2\Lam$ as a translation together with a polarization.

(iii) In contrast to the action $\bullet$, the action $\cdot$ on
$\bc$ given in (\ref{GeoAction}) has this ``nice'' property that for
any $\a\in \Pi_0$, $w_\a$ takes any simplex in $\bc$ to another
simplex in $\bc$, topologically connected to it. This has been our
main reason for choosing the action $\cdot$ instead of $\bullet$.
This completes our remark.
\end{rem}

For $\a_1,\dots,\a_k\in R_t^\times$, we define {\it the path corresponding to} $w_{\a_1}\cdots w_{\a_k}$ in $\bc$ starting at $B_{\sg,\eta}$
to be the $(k+1)$-tuple
$$(B_{\sg,\eta},w_{\a_k}\cdot B_{\sg,\eta},\dots,w_{\a_1}\cdots w_{\a_k}\cdot B_{\sg,\eta}).$$
We denote this path by $P_{B_{\sg,\eta}}(w_{\a_1}\cdots w_{\a_k})$, and we say that this is a path of {\it length} $k$. If $\sg=0$ and $\eta=1$,
we simply denote this path by $P(w_{\a_1}\cdots w_{\a_k})$.
We call a path a \textit{loop based at}  $B_{\sg,\eta}$,
if its starting and ending points are $B_{\sg,\eta}$.
We call the single tuple $(B_{\sg,\eta})$, the {\it trivial loop based at} $B_{\sg,\eta}$ and denote it by $P_{B_{\sg,\eta}}(1)$. Since $\w$ acts freely on $\bc$,
the path $P_{B_{\sg,\eta}}(w_{\a_1}\cdots w_{\a_k})$ is a loop if and only if $w_{\a_1}\cdots w_{\a_k}=1$.
Thus, by Proposition \ref{cor1}, the path $P_{B_{\sg,\eta}}(w_{\a_1}\cdots w_{\a_k})$
is a loop if and only if $(\a_1,\dots,\a_k)\in\alt(R_t^\times)$. Finally, note that $\w$ acts on the set of paths in $\bc$, corresponding to
$w_{\a_1}\cdots w_{\a_k}\in\w$, by
$$w\cdot P_{B_{\sg,\eta}}(w_{\a_1}\cdots w_{\a_k}):=P_{w\cdot B_{\sg,\eta}}(w_{\a_1}\cdots w_{\a_k}),\quad (w\in\w).$$

Let $\pc$ be the set of all paths corresponding to $w_{\ep+\sg_i}^2$, for $0\leq i\leq\nu$, and $(w_\ep w_{\ep+\sg_i}w_{\ep+\sg_j})^2$, for $1\leq i<j\leq\nu$. Any path in $\pc$ is a loop of length either 2 or 6.

\begin{DEF}\label{paths}
(i) Let $P=(B_{\sg_1,\eta_1},\dots,B_{\sg_n,\eta_n})$ be a path in $\bc$. For $1\leq i<j\leq n$,
we call $(B_{\sg_i,\eta_i},\dots,B_{\sg_j,\eta_j})$ a sub-path of $P$. The trivial loops $(B_{\sg_i,\eta_i})$, for $1\leq i\leq n$, are called {\it trivial sub-loops}.

(ii) Let $P_1=(B_{\sg_1,\eta_1},\dots,B_{\sg_n,\eta_n})$ and $P_2=(B_{\sg_n,\eta_n},\dots,B_{\sg_k,\eta_k})$ be two paths in $\bc$,
where the ending point of $P_1$ is the same as the starting point of $P_2$.
We define $$P_1\cdot P_2:=(B_{\sg_1,\eta_1},\dots,B_{\sg_n,\eta_n},\dots,B_{\sg_k,\eta_k}).$$

(iii) Let $P_1$ and $P_2$ be two paths in $\bc$. A {\it move} of $P_1$ is obtained either by replacing a trivial sub-loop based at $B_{\sg,\eta}$
with a loop based at $B_{\sg,\eta}$ from $\pc$, or by replacing a sub-loop based at $B_{\sg,\eta}$ which is an element of $\pc$
with the trivial loop $(B_{\sg,\eta})$.
We say that $P_1$ can be {\it moved} to $P_2$ and we write $P_1\rightarrow P_2$, if $P_2$ is obtained from $P_1$ by a finite number of moves.
\end{DEF}

It is easy to see that
if $P_{B_{\sg,\eta}}(w_{\a_1}\cdots w_{\a_k})\rightarrow P_{B_{\sg,\eta}}(w_{\b_1}\cdots w_{\b_n})$
then
\begin{equation}\label{pathact}
w\cdot P_{B_{\sg,\eta}}(w_{\a_1}\cdots w_{\a_k})\rightarrow w\cdot P_{B_{\sg,\eta}}(w_{\b_1}\cdots w_{\b_n}),
\end{equation}
and
\begin{equation}\label{movepro}
P_{B_{\sg,\eta}}(w_{\a_1}\cdots w_{\a_k})\cdot P'\rightarrow P_{B_{\sg,\eta}}(w_{\b_1}\cdots w_{\b_n})\cdot P',
\end{equation}
for any $w\in\w$ and any path $P'$ in $\bc$ for which the product on the left of (\ref{movepro}) is defined.

Recall that $\w=\w_b=\la w_\a\mid\a\in\Pi_0=\{\ep,\ep+\sg_1,\dots,\ep+\sg_\nu\}\ra$.
%
%
%

\begin{thm}\label{Baby W pre-geo}
Let $R$ be an extended affine root system of type $A_1$ and nullity $\nu$. Then the Weyl group $\w$ of $R$ is isomorphic to the group $G$, defined by
\begin{itemize}
\item[] generators: $x_i$, $0\leq i\leq\nu$,
\item[] relations:
$x^2_k$, $(x_0x_ix_j)^2$, $0\leq k\leq\nu$, $1\leq i<j\leq\nu$.
\end{itemize}
\end{thm}
\begin{proof}
By Theorem \ref{W(V) AP}, $\w$ has the presentation
\begin{itemize}
\item[] generators: $w_\alpha$, $\alpha\in\Pi_0$,
\item[] relations: $w_{\alpha_1}\cdots w_{\alpha_k}$, $(\a_1,\ldots,\a_k)\in\alt(\Pi_0)$.
\end{itemize}
It means that every relation in $\w$ corresponds to a loop in $\bc$. Let us for $0\leq i\leq\nu$ denote $w_{\ep+\sg_i}$ by $x_i$. Then the theorem is proved
 if we show that, for $\sg\in\Lam$ and $\eta\in\{\pm1\}$, any loop based at $B_{\sg,\eta}$ corresponding to $x_{i_1}\cdots x_{i_k}$ can be moved to the trivial loop $(B_{\sg,\eta})$. Now, using (\ref{pathact}) and the fact that action of $\w$ on $\bc$ is transitive, it is enough to show that any loop based at $B_{0,1}$
  can be moved to the trivial loop $(B_{0,1})$. We show this by induction on the length $2m$ of a loop based at $(B_{0,1})$.

Unless otherwise mentioned, all loops are considered based at $B_{0,1}$. First, we show that the assertion holds for $m=1,2,3$.
Let $m=1$. From Remark \ref{eaalt} (iii), we know that any loop of length $2$ corresponds to $x_i^2$, for $0\leq i\leq\nu$.
Thus any loop of length $2$ is an element of $\pc$ and so by definition can be moved to the trivial loop. Next,
let $m=2$. By Remark \ref{eaalt} (iii), a loop of length 4 is either of the form $P(x_i^2x_j^2)$ or $P(x_ix_j^2x_i),$
$0\leq i,j\leq\nu$. Each of these loops  can be moved to the trivial loop as follows:
$$P(x_i^2x_j^2)=P(x_j^2)\cdot P(x_i^2)\rightarrow P(x_j^2)\rightarrow P(1)$$ and $$P(x_ix_j^2x_i)=P(x_i)\cdot P_{x_i\cdot B_{0,1}}(x_j^2)\cdot P_{x_i\cdot B_{0,1}}(x_i)\rightarrow P(x_i^2)\rightarrow P(1).$$
Let $m=3$ and consider a loop corresponding to $x=x_{j_1}\cdots x_{j_6}$. Then $(\a_{j_1},\dots,\a_{j_6})\in\alt(\Pi_0)$. If
$(\a_{j_1},\dots,\a_{j_6})$ contains an alternating 4-tuple
$f'=(\a_{j_s},\ldots,\a_{j_{s+3}})$ and $y$ is the element in
$\w$ corresponding to $f'$, then $x$ has to have one of the
forms $x_j^2y$, $x_jyx_j$ or $yx_j^2$, so by the cases $m=1,2$,
the loop corresponding to $x$ can be moved to the trivial loop. So we may assume that
$(\a_{j_1},\ldots,\a_{j_6})$ contains no alternating 4-tuple
$f'$ as above.
By our assumption we know that the loop corresponding to $(x_0x_ix_j)^2$, for $1\leq i<j\leq\nu$, is an element of $\pc$, so by definition
it can be moved to the trivial loop. Now, consider $(x_jx_ix_0)^2$, the inverse of $(x_0x_ix_j)^2$.
Since $P((x_0x_ix_j)^2)\rightarrow P(1)$, using (\ref{movepro}), we have
\begin{eqnarray*}
P(x_0x_ix_jx_0x_i)&\rightarrow&P(x_j^2)\cdot P(x_0x_ix_jx_0x_i)\\
&=&P(x_j)\cdot P_{x_j\cdot B_{0,1}}((x_0x_ix_j)^2)\\
&\rightarrow&P(x_j).
\end{eqnarray*}
By repeating this process, we obtain $P((x_jx_ix_0)^2)\rightarrow P(1)$.
Using similar arguments, we conclude that
\begin{equation}\label{atleat}
\hbox{if }0\in\{r,s,t\}\quad\hbox{then}\quad P((x_rx_sx_t)^2)\rightarrow P(1).
\end{equation}
 Also if $0\in\{r,s,t\}$, we have
\begin{equation}\label{halfloop1}
\begin{split}
P(x_rx_sx_t)&\rightarrow P((x_tx_sx_t)^2)\cdot P(x_rx_sx_t)\\
&=P(x_tx_sx_r)\cdot P_{x_tx_sx_rB_{0,1}}(x_rx_sx_t^2x_sx_r)\\
&\rightarrow P(x_tx_sx_t).\\
\end{split}
\end{equation}
Now to finish the case $m=3$, consider the element $x=(x_rx_sx_t)^2$, where none of $r,s$ and $t$ is zero. We have
\begin{eqnarray*}
P(x)&\rightarrow& P(x)\cdot P(x_0^2)\\
&=&P(x_tx_rx_sx_t)\cdot P_{x_tx_rx_sx_t\cdot B_{0,1}}(x_sx_rx_0)\cdot P_{x_sx_rx_0x_tx_rx_sx_t\cdot B_{0,1}}(x_0)\\
&\rightarrow&P(x_sx_t)\cdot P_{x_sx_t\cdot B_{0.1}}(x_0^2)\cdot P_{x_sx_t\cdot B_{0,1}}(x_rx_0x_tx_r)\\
& &\cdot P_{x_rx_0x_tx_rx_sx_t\cdot B_{0,1}}(x_t^2)\cdot P_{x_rx_0x_tx_rx_sx_t\cdot B_{0,1}}(x_0x_s)\\
&=&P(x_0x_sx_t)\cdot P_{x_0x_sx_t\cdot B_{0,1}}((x_tx_rx_0)^2)\cdot P_{x_0x_sx_t\cdot B_{0,1}}(x_0x_sx_t)\\
&\rightarrow& P(x_0x_sx_t)\cdot P_{x_0x_sx_t\cdot B_{0,1}}(x_0x_sx_t)\\
&=&P((x_0x_sx_t)^2)\\
&\rightarrow&P(1).
\end{eqnarray*}
Thus any loop of length 6 corresponding to expressions with respect to $\Pi_0$ can be moved to the trivial loop.
Also using the same argument as in (\ref{halfloop1}), we get
\begin{equation}\label{halfloop}
P(x_rx_sx_t)\rightarrow P(x_tx_sx_r)\quad\quad(0\leq r,s,t\leq\nu).
\end{equation}

Now, we assume that $m>3$ and that any loop of length smaller than $2m$ can be moved to a trivial loop. Let $P(x)$ be a loop of length $2m$, where $x=x_{i_1}\cdots x_{i_{2m}}$. First assume that for some $1\leq r\leq 2m-1$, $j_r=j_{r+1}$.
Since $P_{B_{\sg,p}}(x_{j_r}^2)\rightarrow P_{B_{\sg,p}}(1)$, we have
$$P(x)\rightarrow P(x_{j_1}\cdots x_{j_{r-1}} x_{j_{r+2}} \cdots x_{j_{2m}}).$$
Now since $(\a_{j_1},\dots,\a_{j_{r-1}},\a_{j_{r+2}},\dots,\a_{j_{2m}})$
is an alternating $(2m-2)$-tuple,  $P(x_{\a_{j_1}}$ $\cdots x_{\a_{j_{r-1}}} x_{\a_{j_{r+2}}} \cdots x_{\a_{j_{2m}}})$
is a loop of length $2m-2$ in $\bc$ and so by induction steps it can be moved to the trivial loop. So, we may assume that
$j_r\not=j_{r+1}$, for all $1\leq r\leq 2m-1$. From Remark
\ref{eaalt}(ii), the root $\a_{j_1}$ appears in
$(\a_{j_1},\dots,\a_{j_{2m}})$ in an even number of times and there is an even
integer $2\leq s\leq 2m$ such that $j_1=j_s$. From (\ref{halfloop}), we
have
$$P(x)\rightarrow P(x_{j_3}x_{j_2}x_{j_1}x_{j_4}\cdots x_{j_{2m}}).$$
By repeating this process, we can move $x_{j_1}$ adjacent to $x_{j_s}$, namely
$$P(x)\rightarrow P(x_{j_3}x_{j_2}x_{j_5}x_{j_4}\cdots x_{j_{s-1}}x_{j_{s-2}}x_{j_1}x_{j_s}\cdots x_{j_{2m}}).$$
By Remark \ref{eaalt}(i), the $2m$-tuple
$$(\a_{j_3},\a_{j_2},\a_{j_5},\a_{j_4}\dots \a_{j_{s-1}},\a_{j_{s-2}},\a_{j_1},\a_{j_s},\dots,\a_{j_{2m}})$$ is alternating. Thus $(\a_{j_3},\a_{j_2},\a_{j_5},\a_{j_4}\dots \a_{j_{s-1}},\a_{j_{s-2}},\a_{j_{s+1}},\dots,\a_{j_{2m}})$ is an alternating  $(2m-2)$-tuple and
$$P(x)\rightarrow P(x_{j_3}x_{j_2}x_{j_5}x_{j_4}\cdots x_{j_{s-1}}x_{j_{s-2}}x_{j_{s+1}}\cdots x_{j_{2m}}).$$
By induction steps, the right hand side, which is a loop in $\bc$ of length $2m-2$, can be moved to the trivial loop. Thus every loop in $\bc$ can be moved to the trivial loop.
\end{proof}

We conclude this section with the following example which gives a geometric illustration of the method we used in the proof of  Theorem \ref{Baby W pre-geo}. We use the same notation as in Sections \ref{RedPre} and \ref{geometric}.
\begin{exa}
Let $\v^0=\mathbb{R}\sg_1\oplus\mathbb{R}\sg_2$ and
$w=w_{\ep+\sg_2}w_{\ep}w_{\ep+\sg_2}w_{\ep+\sg_1}w_{\ep}w_{\ep+\sg_1}w_{\ep}w_{\ep+\sg_2}$ $w_{\ep+\sg_1}w_{\ep+\sg_2}w_{\ep+\sg_1}w_{\ep}.$ Since $(\sg_2,0,\sg_2,\sg_1,0,\sg_1,0,\sg_2,\sg_1,\sg_2,\sg_1,0)$ is an alternating $12$-tuple, $w$ is a relation in $\w$. Now, we use our approach to illustrate geometrically how the path $P(w)$ can be moved to the
trivial path $P(1)$. With the notations of Theorem \ref{Baby W pre-geo}, we write $w=x_2x_0x_2x_1x_0x_1x_0x_2x_1x_2x_1x_0$, where $x_i=w_{\ep+\sg_i}$ for $0\leq i\leq\nu$. We have
\begin{eqnarray*}
P(w)=(B_{0,1},B_{0,-1},B_{-\sg_1,1},B_{\sg_2-\sg_1,-1},B_{\sg_2-2\sg_1,1},B_{2\sg_2-2\sg_1,-1},B_{2\sg_2-2\sg_1,1},\;\;\\
B_{2\sg_2-\sg_1,-1},B_{2\sg_2-\sg_1,1},B_{2\sg_2,-1},B_{\sg_2,1},B_{\sg_2,-1},B_{0,1}).
\end{eqnarray*}
\begin{figure}[!hpp]
\caption{The complex $P(w)$ in $\v^0$}\label{pic1}
\setlength{\unitlength}{10pt}
\begin{picture}(20,15)(0,0)
\put(0,4){\vector(1,0){20}}
\put(16,0){\vector(0,1){14}}
\thicklines
\put(12,4){\line(1,0){6}}
\put(16,2){\line(0,1){10}}
\put(20,4.5){\makebox{\small{$\sg_1$}}}
\put(16.5,14){\makebox{\small{$\sg_2$}}}
\put(6,12){\line(1,-1){10}} \put(6,12){\line(1,0){10}}
\put(8,14){\line(0,-1){6}} \put(8,14){\line(1,-1){4}}
\put(12,10){\line(0,1){4}} \put(12,14){\line(1,-1){6}}
\put(18,8){\line(-1,0){4}} \put(14,8){\line(1,-1){4}}
\put(8,8){\line(1,0){4}}\put(12,8){\line(0,-1){4}}
\put(16.2,4.2){\makebox{\small{$B_0$}}}
\put(15,3){\makebox{\small{$B_1$}}}
\put(12.2,4.2){\makebox{\small{$B_2$}}}
\put(10.9,7){\makebox{\small{$B_3$}}}
\put(8,8.2){\makebox{\small{$B_4$}}}
\put(6.9,11.1){\makebox{\small{$B_5$}}}
\put(8,12.5){\makebox{\small{$B_6$}}}
\put(10.8,11){\makebox{\small{$B_7$}}}
\put(12,12.5){\makebox{\small{$B_8$}}}
\put(14.8,11){\makebox{\small{$B_9$}}}
\put(16,8.5){\makebox{\small{$B_{10}$}}}
\put(14.6,7.2){\makebox{\small{$B_{11}$}}}
\end{picture}
\end{figure}

In the first move, we have
$$P(w)\longrightarrow P(x_2x_0x_0x_1x_2x_1x_0x_2x_1x_2x_1x_0)=P(x_2x_1x_2x_1x_0x_2x_1x_2x_1x_0).$$
Let
$w_1:=x_2x_0x_0x_1x_2x_1x_0x_2x_1x_2x_1x_0=x_2x_1x_2x_1x_0x_2x_1x_2x_1x_0$.
Then
\begin{eqnarray*}
P(w_1)=(B_{0,1},B_{0,-1},B_{-\sg_1,1},B_{\sg_2-\sg_1,-1},B_{\sg_2-2\sg_1,1},B_{2\sg_2-2\sg_1,-1},B_{2\sg_2-2\sg_1,1},\;\;\\
B_{2\sg_2-\sg_1,-1},B_{\sg_2-\sg_1,1},B_{\sg_2,-1},B_{0,1}).
\end{eqnarray*}
\begin{figure}
\caption{The complex $P(w_1)$ in $\v^0$}\label{pic2}
\setlength{\unitlength}{10pt}
\begin{picture}(20,15)(0,0)
\put(0,4){\vector(1,0){20}}
\put(16,0){\vector(0,1){14}}
\thicklines
\put(12,4){\line(1,0){6}}
\put(16,2){\line(0,1){6}}
\put(20,4.5){\makebox{\small{$\sg_1$}}}
\put(16.5,14){\makebox{\small{$\sg_2$}}}
\put(6,12){\line(1,-1){10}} \put(6,12){\line(1,0){6}}
\put(8,14){\line(1,-1){10}} \put(8,14){\line(0,-1){6}}
\put(8,14){\line(1,-1){4}}\put(12,12){\line(0,-1){8}}
\put(16,8){\line(-1,0){8}} \put(16.2,4.2){\makebox{\small{$B_0$}}}
\put(15,3){\makebox{\small{$B_1$}}}
\put(12.2,4.2){\makebox{\small{$B_2$}}}
\put(10.9,7){\makebox{\small{$B_3$}}}
\put(8,8.2){\makebox{\small{$B_4$}}}
\put(6.9,11.1){\makebox{\small{$B_5$}}}
\put(8,12.5){\makebox{\small{$B_6$}}}
\put(10.8,11){\makebox{\small{$B_7$}}}
\put(12.2,8.3){\makebox{\small{$B_8$}}}
\put(14.6,7.2){\makebox{\small{$B_{9}$}}}
\end{picture}
\end{figure}

\noindent As one can see, the first move replaces the path $(B_7,B_8,B_9,B_{10},B_{11})$ in Figure \ref{pic1} by the path $(B_7,B_8,B_9)$ in Figure \ref{pic2}.

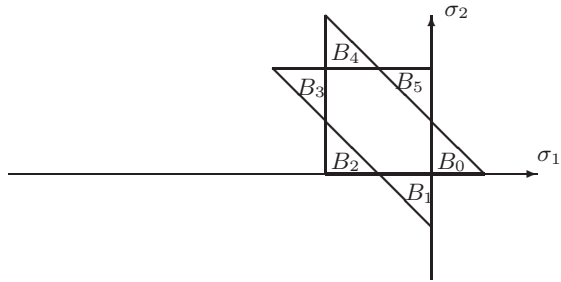
\begin{figure}
\caption{The complex $P(w_2)$ in $\v^0$}\label{pic3}
\begin{picture}(200,120)(0,0)\label{pic3}
\put(0,40){\vector(1,0){200}}
\put(160,0){\vector(0,1){100}}
\thicklines
\put(120,40){\line(1,0){60}}
\put(160,20){\line(0,1){60}}
\put(200,45){\makebox{\small{$\sg_1$}}}
\put(165,100){\makebox{\small{$\sg_2$}}}
\put(160,20){\line(-1,1){60}} \put(180,40){\line(-1,1){60}}
\put(120,40){\line(0,1){60}} \put(160,80){\line(-1,0){60}}
\put(162,42){\makebox{\small{$B_0$}}}
\put(150,30){\makebox{\small{$B_1$}}}
\put(122,42){\makebox{\small{$B_2$}}}
\put(109,70){\makebox{\small{$B_3$}}}
\put(122,83){\makebox{\small{$B_4$}}}
\put(146,72){\makebox{\small{$B_5$}}}
\end{picture}
\end{figure}

In the second move, we have
$$P(w_1)\longrightarrow
P(x_2x_1x_0x_1x_2x_2x_1x_2x_1x_0)=P((x_2x_1x_0)^2).$$ Let
$w_2:=(x_2x_1x_0)^2$. Then
$P(w_2)=(B_{0,1},B_{0,-1},B_{-\sg_1,1},B_{\sg_2-\sg_1,-1},B_{\sg_2-\sg_1,1},B_{\sg_2,-1},$ $B_{0,1}).$

\noindent The second move replaces the path $(B_3,B_4,B_5,B_6,B_7,B_8)$ in Figure \ref{pic2} with the path $(B_3,B_4)$ in Figure \ref{pic3}. Since $w_2$ belongs to $\mathcal{P}$, from Definition \ref{paths}(iii), $P(w_2)$ moves to $P(1)$.
\end{exa}
\centerline{\bf Acknowledgements}
{This research was in part supported by a grant from IPM (No. 90170217).
The authors would like to thank  the Center of Excellence for Mathematics, University of Isfahan.} The authors also would like to thank the anonymous referee for several crucial comments on an early version
of this work and also for
suggesting a geometric approach to the proof of Theorem \ref{Baby W pre}.

\end{document}